\documentclass[12pt,reqno]{amsart}
\usepackage[comma,numbers]{natbib}
\usepackage[a4paper,left=1in,right=1in,top=1.25in,bottom=1.25in]{geometry}
\usepackage{color}
\usepackage{graphicx}
\usepackage{amsmath}

\newtheorem{theorem}{Theorem}[section]
\newtheorem{corollary}[theorem]{Corollary}
\newtheorem{lemma}[theorem]{Lemma}
\theoremstyle{definition}
\newtheorem{defn}{Definition}[section]
\theoremstyle{remark}

\def\numberlikeadb{\global\def\theequation{\thesection.\arabic{equation}}}
\numberlikeadb

\newcommand{\eqa}{\begin{eqnarray}}
\newcommand{\ena}{\end{eqnarray}}
\newcommand{\eq}{\begin{equation}}
\newcommand{\en}{\end{equation}}
\newcommand{\eqs}{\begin{eqnarray*}}
\newcommand{\ens}{\end{eqnarray*}}

\def\bone{{\bf 1}}
\def\th{{\Theta}}
\def\om{{\Omega}}
\def\thom{{\Theta\times\Omega}}

\begin{document}
\title[Metapopulation with Markovian landscapes]{A metapopulation model with Markovian landscape dynamic}
\thanks{PKP and RM are supported in part by the Australian Research Council (Discovery Grant DP150101459 and the ARC Centre of Excellence for Mathematical and Statistical Frontiers, CE140100049). \\ \copyright 2016. This manuscript version is made available under the CC-BY-NC-ND 4.0 license http://creativecommons.org/licenses/by-nc-nd/4.0/}
\maketitle
\noindent R. McVINISH\footnote{Corresponding author: email r.mcvinish@uq.edu.au}, P.K. POLLETT and Y.S. CHAN \\
School of Mathematics and Physics, University of Queensland \\

\noindent ABSTRACT. We study a variant of Hanski's incidence function
model that allows habitat patch characteristics to vary over time
following a Markov process. The widely studied case where patches are
classified as either suitable or unsuitable is included as a special
case. For large metapopulations, we determine a recursion for the
probability that a given habitat patch is occupied.
This recursion enables us to clarify the role of
landscape dynamics in the survival of a metapopulation. In particular,
we show that landscape dynamics affects the persistence and equilibrium level of the metapopulation primarily through its effect on the distribution of a local population's life span.


\section{Introduction}

A metapopulation is a collection of local populations of a single focal species occupying spatially distinct habitat patches. Much of the research on metapopulations has focussed on identifying and quantifying extinction risks, with mathematical modelling playing an important role. Levins \cite{Levins:69} proposed the first model of a metapopulation, which, despite its many simplifying assumptions, provided a number of important insights \cite{Hanski:91}. The importance of spatial features such as landscape heterogeneity and patch connectivity to metapopulation persistence was demonstrated in subsequent research \cite{Hanski:94,OH:01,MN:02,HO:03}  and by connections to interacting particle systems \cite{Liggett:05,DL:94,Franc:04}. Of particular relevance to the current work is the Incidence Function Model (IFM) \citep{Hanski:94}, which relates the colonisation and local extinction probabilities to landscape characteristics. This model has been used to study extinction risk and the effectiveness of conservation measures for a number of populations including the African lion (\emph{Panthera leo}) in Kenya and Tanzania \citep{DSHLF:14}, the water vole (\emph{Arvicola amphibius}) in the UK \citep{MB:11}, and the prairie dog (\emph{Cynomys ludovicianus}) in northern Colorado, USA \citep{GWPSA:13}.

While real landscapes are structured spatially, they also vary temporally. Landscape dynamics are known to play an important role in the persistence/extinction of a number of species \citep{vTVO:12}. As an example, \citet{Hanski:99} mentions the marsh fritillary butterfly (\emph{Eurodryas aurinia}) whose host plant \emph{Succisa pratensis} occurs in forest clearings that are between two and ten years old. The metapopulation of sharp-tailed grouse (\emph{Tympanuchus phasianellus}), which occupies areas of grassland, is similarly affected by landscape dynamics \citep{ARMH:04, GN:00}. For this species, fire opens new grassland areas and prevents the encroachment of forests. Other examples include metapopulations of the perennial herb \emph{Polygonella basiramia} \citep{BMW:03} and metapopulations of the beetle \emph{Stephanopachys linearis}, which breeds only in burned trees \citep{RBHWC:14}. In these examples, the landscape dynamics are driven by secondary succession, and this is often the case regardless of whether the focal species is part of a seral community or the climax community. 

Some authors \citep{BOGKFLG:99,VVVCPH:04,JRS:12} have attempted to deal with landscape dynamics by incorporating the time elapsed since the patch was created through the local extinction probability.   A more widely used and studied approach incorporates landscape dynamics by allowing each patch to alternate between being suitable or unsuitable for supporting a local population. In its simplest form all patches are treated equally \cite{KMVHL:00,Ross:06,WCP:06,RSAH:15}. A more general form used by \citet{DFS:05} and \citet{XFAS:06} incorporates differences between patches in area and extinction rates. These studies demonstrate an important relationship between the time scale of metapopulation dynamics and landscape dynamics. When the habitat life span is too short, the metapopulation is unable to become established \cite{KMVHL:00,DFS:05}. Furthermore, ignoring landscape dynamics leads to inaccurate persistence criteria. In general, the persistence criteria is optimistic \cite{Ross:06,XFAS:06}, though not necessarily for species which are able to react to habitat destruction \cite{RSAH:15}.

The main problem we see with the suitable/unsuitable classification is that it is too coarse. Treating patches as being unsuitable may be reasonable following a destructive event, but this approach is unable to handle typical environmental fluctuations in habitat size or quality. We note that earlier examples of modelling change in natural landscapes with Markov chains used larger state spaces \cite[Table 2]{Baker:89}. A related but more subtle problem is that all two-state Markov chains, like those used to model landscape dynamics in the above papers, are reversible in the sense that the process appears the same when time is reversed \cite[section 1.2]{Kelly:79}. The process of landscape succession following disturbance typically proceeds through a number of stages in a fixed order \cite{PLD:16}. This is inconsistent with reversibility since for a reversible Markov chain any path which ultimately returns to the starting state must have the same probability regardless of whether this is traced in one direction or the other \cite[section 1.5]{Kelly:79}. Furthermore, the transition kernel of a reversible Markov chain has a special structure which most Markov chains do not possess \cite[Theorem 1.7]{Kelly:79}.

In this paper we adopt a similar approach to the one proposed by \citet{Hanski:99} and incorporate landscape dynamics into the incidence function model by treating landscape characteristics as a stochastic process. To the authors' knowledge, this approach has not been subject to a detailed mathematical analysis. Specifically, we model the landscape characteristics  as a Markov chain on a general state space. The Markov chain model facilitates the analysis while allowing a general state space avoids the issues raised with the suitable/unsuitable classification. Our primary aim is to understand the affect of landscape dynamics on the survival of the metapopulation. Previous studies using the suitable/unsuitable classification have shown that landscape dynamics affect the survival of large metapopulations only through the average life span of the local populations \cite{DFS:05,XFAS:06}. It is natural to consider whether this behaviour holds for more general landscape dynamics. In order to address this question, we study the limiting behaviour of the metapopulation when the number of patches is large. Using an analysis similar to our earlier work \citep{MP:12,MP:13,MP:14}, we show large metapopulations display a  deterministic limit and asymptotic independence of local populations. All proofs are given in the Appendix.

\section{Model description} \label{Sec:MD}

Hanski's IFM \citep{Hanski:94} describes the evolution of a metapopulation as a discrete-time Markov chain. In this model, the dynamics of the metapopulation are determined by characteristics of the habitat patches. These characteristics are the patch's location, a weight related to the size of the patch, and the probability that a local population occupying this patch survives a given period of time. The colonisation and extinction processes which govern the metapopulation are specified by these characteristics. The local population of a patch will go extinct at the next time step with probability given by one minus the local survival probability. A patch that is currently empty will be colonised by individuals migrating from the occupied patches at the next time step with a probability that depends on its relative location to the occupied patches and the weight associated with those occupied patches. For a metapopulation comprising $ n $ patches, the presence or absence of the focal species at each of the patches determines the metapopulation's state.

Before describing the metapopulation dynamics more formally, we need to introduce some notation. The state of the metapopulation at time~$ t $ is denoted by the binary vector $ X^{n}_{t} = (X_{1,t}^{n},\ldots,X_{n,t}^{n})$, where $ X_{i,t}^{n} = 1 $ if patch $
i $ is occupied at time~$ t $ and $ X_{i,t}^{n} = 0 $ otherwise. The patch location, weight and survival probability of patch $ i $ are denoted by $ z_{i},\ A_{i} $ and $ s_{i} $, respectively. We let $ z^{n},\ A^{n} $ and $ s^{n} $ denote the respective vectors of all $ n $ patch locations, weights and survival probabilities. Since we view the landscape as the result of some random process, we treat $ z^{n},\ A^{n} $ and $ s^{n} $ as random vectors. Conditional on the state of the metapopulation at time $ t$  and the patch characteristics, the $ X_{i,t+1}^{n} \ (i=1,\ldots,n) $ are independent with transition probabilities 
\begin{equation}
\mathbb{P}\left(X_{i,t+1}^{n}=1 \mid X_{t}^{n},z^{n},A^{n},s^{n}\right)  = s_{i}X^{n}_{i,t} + f\left(\sum_{j=1}^{n} X_{j,t}^{n} D(z_{i},z_{j}) A_{j}\right)\left(1-X_{i,t}^{n}\right), \label{MD:Eq1}
\end{equation}
where $D(z_{i},z_{j}) = \exp(-\alpha \|z_{i} - z_{j}\|) $ for some $ \alpha > 0 $ and the function $ f :[0,\infty) \rightarrow [0,1]$ is the colonisation function. \citet{Hanski:94} takes the colonisation function to be $ f(x) = \beta x^{2}/(\gamma + \beta x^{2}) $ for some $ \beta,
\gamma > 0 $ and assumes the survival probability $ s_{i} $ is a function of $ A_{i} $. The sum $ \sum_{j=1}^{n} X_{j,t}^{n} D(z_{i},z_{j}) A_{j} $ is called the connectivity measure. It specifies how each occupied patch contributes to the potential colonisation of any given unoccupied patch $i$. Larger patches make a greater contribution to the connectivity measure as they are expected to have a greater capacity to produce propagules. These propagules then disperse to other patches based on their proximity so nearby patches contribute more to the connectivity measure.

Before incorporating landscape dynamics into this model, it will be useful to briefly discuss our approach to the analysis. The central idea is that as the number of habitat patches increases, the connectivity measure at each location converges to a deterministic limit in probability. This is based on scaling  the connectivity measure with the number of patches in the metapopulation and can be achieved in a number of ways. One way is if the total habitable area for the focal species is fixed. As the number of patches increases, the landscape becomes more fragmented. If all patches are of a comparable size, then $ A_{i} $ should be of the order $ n^{-1} $ and the connectivity measure for patch $ i $ can be expressed as
\begin{equation}
n^{-1} \sum_{j=1}^{n}  X_{j,t}^{n} D(z_{i},z_{j}) a_{j}, \label{MD:Con:Eq1}
\end{equation}
where $ A_{j} = n^{-1} a_{j} $. This type of scaling is discussed in \citet[section 5]{BMP:15} in connection with approximating the stochastic model (\ref{MD:Eq1}) by a deterministic difference equation. Alternatively, we might consider that each patch has a finite number of propagules that can be dispersed in a given period. These propagules are divided between the $ n -1 $ potential destination patches based on their proximity. As $ n \rightarrow \infty $, the propagules are divided between more patches. This introduces a factor of $ n^{-1} $ in the connectivity measure resulting in (\ref{MD:Con:Eq1}). Finally, in a similar spirit to \cite{OC:06}, we could consider an increasing region with constant density of patch locations and scale the dispersal kernel so that $ D_{n}(z_{i},z_{j}) = \alpha_{n}^{d}\exp(-\alpha_{n}\|z_{i} - z_{j}\|) $, with $ \alpha_{n} \rightarrow 0 $. Provided the rates at which the region increases and $ \alpha_{n} $ `balance', the same limiting model will be obtained.  Other scalings that result in a `weak law of large numbers' for the connectivity measure, that is convergence in probability of the connectivity measure to a deterministic quantity, will lead to a similar limiting process.

We incorporate landscape dynamics into the incidence function model assuming that only the survival probabilities $ s_{i} $ and the patch weights $ a_{i} $ evolve over time and the patch locations remain static. The survival probability and weight of a patch are modelled using a single variable $ \theta $, which  we call the {\em characteristic\/} of the patch, taking values in some set $ \Theta $. If the characteristic of patch $ i $ at time $ t $ is $ \theta_{i,t} $, then the survival probability and weight are given by $ s(\theta_{i,t}) $ and $ a(\theta_{i,t}) $ respectively, where $ s: \Theta \rightarrow [0,1] $ and $ a:\Theta \rightarrow [0,\infty) $. In this way, we allow dependence between the local survival probability and patch weight. If $ a(\cdot) $ is an invertible function, then we can express the local survival probabilities as a function of patch weights as in \citet{Hanski:94}. Conditional on $ X_{t}^{n}$, $\theta^{n}_{t} :=(\theta_{1,t},\ldots,\theta_{n,t}) $ and $ z^{n} $, the $ X_{i,t+1}^{n} \ (i=1,\ldots,n) $ are independent with transition probabilities 
\begin{equation}
\mathbb{P}\left(X_{i,t+1}^{n}=1 \mid X_{t}^{n},\theta^{n}_{t},z^{n}\right)  = s(\theta_{i,t})X^{n}_{i,t} + f\left(n^{-1}\sum_{j=1}^{n} X_{j,t}^{n} D(z_{i},z_{j})   a(\theta_{j,t})\right)\left(1-X_{i,t}^{n}\right). \label{MD:Eq2}
\end{equation}

Equation (\ref{MD:Eq2}) is a natural extension of Hanski's incidence function model allowing for a dynamic landscape. However, the form of the colonisation probability means that the characteristic of patch $ i $ does not affect the probability that it is colonised, and this may be undesirable in some applications. \citet{MN:02} consider a connectivity measure where the size of the target patch increases the probability of colonisation. Here we allow the colonisation function to be a function of both the connectivity and the characteristic of patch to be colonised, that is,
\begin{equation}
\mathbb{P}\left(X_{i,t+1}^{n}=1 \mid X_{t}^{n},\theta^{n}_{t},z^{n}\right)  = s(\theta_{i,t})X^{n}_{i,t} + f\left(n^{-1}\sum_{j=1}^{n} X_{j,t}^{n} D(z_{i},z_{j})   a(\theta_{j,t}); \theta_{i,t}\right)\left(1-X_{i,t}^{n}\right). \label{MD:Eq2b}
\end{equation}
A model with the connectivity measure proposed by \citet{MN:02} is obtained by setting $ f(x,\theta) = \bar{f}(b(\theta)x) $ for some functions $ \bar{f}: [0,\infty) \rightarrow [0,1] $ and $ b:\Theta \rightarrow [0,\infty) $. Perhaps more importantly, model (\ref{MD:Eq2b}) allows for two effects that are not possible with model (\ref{MD:Eq2}); phase structure and pulsed dispersal.

Suppose the colonisation and extinction phases alternate, with observations of the metapopulation made after the extinction phase. In this case the colonisation probability has the form $ f(x,\theta) = s(\theta) \bar{f}(x,\theta) $ where $ \bar{f}: [0,\infty) \times \Theta \rightarrow [0,1] $. This type of phase structure has previously been used in \cite{AG:91,DP:95,HC:01,MP:13}. We note that if observations were instead taken after the colonisation phase, then the model would display the rescue effect \cite{Hanski:94}.

Pulsed dispersal occurs when migration from a colonised patch is the result of the species response to a decline in habitat quality. A continuous time metapopulation model with suitable/unsuitable landscape dynamics is studied in \citep{RSAH:15}. Suppose for simplicity that $ \Theta = \{1,2\} $, where the state $ 1 $ indicates a suitable habitat patch and state $ 2 $ indicates an unsuitable habitat patch. Pulsed dispersal is achieved by taking $ a(2) > a(1) $, that is an occupied patch that has recently become unsuitable will contribute more to the colonisation of other patches than an occupied patch that is still suitable. Taking $ s(2) = f(\cdot,2) = 0 $ ensures the local population at the unsuitable patch becomes extinct with probability one and a patch cannot be recolonised while it is unsuitable.

\section{Asymptotic behaviour of a single patch} \label{Sec:Asym}

The typical approach to incorporating landscape dynamics into a metapopulation model is to allow each patch to alternate between being suitable or unsuitable for supporting a local population according to some Markov chain, independently of the other patches and of the state of the metapopulation \citep{KMVHL:00,Ross:06,WCP:06,XFAS:06,RSAH:15}. A natural extension is to model the temporal evolution of each patch characteristic by a Markov chain on a finite state space $ \{1,\ldots, m\} $  with a common transition probability matrix. However, the classification of habitat into one of a finite number of classes may be unnatural in some settings since, for example, it restricts patch areas to taking  only finitely many values. To avoid this, we model the habitat dynamics as a Markov chain on a general state space with a common transition probability kernel. 

We briefly recall some properties of Markov chains. Suppose $ Y_{t} $ is a Markov chain on the state space $ S $ and let $\Sigma$ be the set of all events of interest (that is, a $\sigma$-field of subsets of $S$). The transition probability kernel $ P $ gives the probability that the chain moves from a point $y$ to the set $ A \in \Sigma $ in one time step:
\begin{equation}
P(y,A) = \mathbb{P} (Y_{t+1} \in A \mid Y_{t} = y). \label{TransKern}
\end{equation}
When the $ \Theta $ is finite, $ \Sigma $ is just the set of all subsets of $ S $ and the right hand side of (\ref{TransKern}) is just $ \sum_{j\in A} P_{yj} $ where $ (P_{ij}) $ is the transition probability matrix of the Markov chain. Under certain weak conditions, the transition kernel has an invariant distribution $ \pi $, that is 
\[
\pi(A) = \int_{\th} P(x,A)\pi(dx), \quad \mbox{for all } A \in \Sigma.
\]
When the $ S $ is finite, so $ P $ is a transition probability matrix, the invariant distribution is simply a distribution $ \pi $ on $ S $ such that $ \pi P = \pi $ and we write $ \pi(A) = \sum_{i \in A} \pi_{i} $. In any case, if $ \pi $ is the distribution of $ Y_{1} $, then $ Y_{t} $ will also have this distribution for all $ t $ and the Markov chain is said to be stationary or in equilibrium. 

Stationary Markov chains with a state space of only two elements  have the rather special property of reversibility. Informally, this means the process will look the same if the direction of time is reversed. More precisely, a stationary Markov chain is reversible with respect to $ \pi $ if
\[
\int_{A} \pi(dx) P(x,B) = \int_{B} \pi(dx) P(x,A), \quad \mbox{for all } A,B \in \Sigma. 
\]
For a finite state space, the condition for reversibility is simply $ \pi_{i} P_{ij} = \pi_{j} P_{ji} $, for all $ i,j \in S $.  

Our analysis makes use of the dual transition kernel, a concept related to reversibility. If  $ P $ has invariant distribution $ \pi $, then there is a transition kernel $ P^{\ast} $ called a \emph{dual of $P$ with respect to} $\pi$  satisfying
\begin{equation}
\int_{A} \pi(dx) P(x,B) = \int_{B} \pi(dx) P^{\ast}(x,A), \quad \mbox{for all } A,B \in \Sigma  \label{Def:Eq1}
\end{equation}
(Theorem \ref{Thm:dual} of Appendix B).  For a finite state space, the dual transition probability matrix is the probability transition matrix satisfying $ \pi_{i} P_{ij} = \pi_{j} P^{\ast}_{ji} $, for all $ i,j \in S $. This expression provides a means of constructing $ P^{\ast} $ from $ P $ and $ \pi $. 

As previously noted, the landscape might be viewed as the result of some
random process. Here we assume that the  $ (\theta_{i,0},z_{i}),\
i=1,\ldots,n $, are independent and identically distributed. The
marginal distribution of the patch locations is assumed to be supported
on $ \Omega \subset \mathbb{R}^{d}$ and have a probability density
function which we denote by $ \zeta $.  Under mild conditions on the
transition kernel of the patch characteristic, landscapes that have
existed for a long time should at least be approximately stationary in
the sense that the distribution of $ \theta_{i,t} $ should converge to
its invariant distribution as $ t \rightarrow \infty $.
For Markov chains on general state spaces, convergence
to the invariant distribution is ensured by the technical condition of
positive Harris recurrence. When the state space is finite, such as for
the suitable/unsuitable classification, positive Harris recurrence holds if the Markov chain is irreducible, that is each state can be reached from any other, and aperiodic. The following result shows this implies independence of the patch location and characteristic. 

\begin{lemma}\label{lemma2}
Suppose that Markov chain $ (\theta_{i,t}, t \geq 0 ) $ taking values in $ \Theta $ is positive Harris and aperiodic with invariant distribution $ \pi $. For any $ i $, the distribution of $ (\theta_{i,t},z_{i}) $ converges as $ t $ goes to infinity. Furthermore, $ \theta_{i,t} $ and $ z_{i} $ are asymptotically independent in the sense that 
\begin{equation}
\lim_{t\rightarrow\infty} \mathbb{P}\left(\theta_{i,t} \in A, z_{i} \in B\right) =\pi\left(A\right)  \int_{B} \zeta(z)(dz), \label{lemma2:Eq1}
\end{equation}
for any measurable $ A \subset \Theta $ and $ B \subset \Omega $.
\end{lemma}

The main result of this section concerns the behaviour of local populations. This result depends on being able to establish convergence in probability of the connectivity measure at each location as the number of habitat patches increases. 
A sequence of random variables $ \{Y^{n}\}_{n=1}^{\infty} $ is said to converge in probability to a random variable $ Y $ if, for any $ \epsilon > 0 $, $ \mathbb{P} ( \| Y^{n} - Y\| \geq \epsilon) \rightarrow 0 $. Convergence in probability is denoted $ Y^{n} \stackrel{p}{\rightarrow} Y $. We are now able to state our results on the behaviour of large metapopulations with Markovian landscape dynamics.

\begin{theorem} \label{thmAsym1}
Suppose that 
\begin{equation}
 \mathbb{P}(X_{i,0}^{n} = 1 \mid \theta_{i,0}=\theta, z_{i} = z) = q_{0}(\theta,z), \label{thmAsym1:eq1}
\end{equation}
for some function $ q_{0} : \Theta\times \Omega \rightarrow [0,1] $. Under the assumptions given in Appendix A, if $ X_{i,0}^{n} \stackrel{p}{\rightarrow} X_{i,0} $, then $ X_{i,t}^{n} \stackrel{p}{\rightarrow} X_{i,t} $ for all $ t \geq 0 $, where the transition probability for $ X_{i,t} $ is 
\begin{equation}
\mathbb{P}\left(X_{i,t+1}=1 \mid X_{i,t},\theta_{i,t} = \theta,z_{i}\right)  = s(\theta)X_{i,t} +  f \left( \psi_{t}(z_{i}); \theta\right) \left(1-X_{i,t}\right) \label{Thm2:Eq0}
\end{equation}
and 
\begin{align}
\psi_{t}(z) & =  \int_\Omega D(z,\tilde{z}) \left(\int_\th a(\theta) q_{t}(\theta,\tilde{z}) \pi(d\theta)\right) \zeta(\tilde{z}) d\tilde{z}, \label{Thm2:Eq1}\\
q_{t+1}(\theta,z) & = \int_\th s(\eta) q_{t}(\eta,z) P^{\ast}(\theta,d\eta)
+ \int_\th f \left(\psi_{t}(z); \eta\right) \left(1 - q_{t}(\eta,z)\right)
P^{\ast}(\theta,d\eta).  \label{Thm2:Eq2}
\end{align}
\end{theorem}

When $ \Theta $ is a finite state space, equations (\ref{Thm2:Eq1}) and (\ref{Thm2:Eq2}) can be expressed as 
\begin{align}
\psi_{t}(z) & =  \int_\Omega D(z,\tilde{z}) \left( \sum_{j} a(j)  q_{t}(j,\tilde{z}) \pi_{j} \right) \zeta(\tilde{z}) d\tilde{z}, \label{Thm1:Eq1}\\
q_{t+1}(i,z) & = \sum_{j} s(j) q_{t}(j,z) P^{\ast}_{ij} + \sum_{j} f \left(\psi_{t}(z); j\right) \left(1 - q_{t}(j,z)\right) P^{\ast}_{ij}. \label{Thm1:Eq2}
\end{align}
The conditions for Theorem \ref{thmAsym1} to hold are given in a very general form in Appendix A. In the context of a finite state space for the landscape characteristic, sufficient conditions for Theorem \ref{thmAsym1} to hold are as follows: (i) the landscape dynamics are irreducible, aperiodic and stationary, (ii) the support for the patch locations is bounded, and (iii) the colonisation function is Lipschitz continuous for each state of the landscape characteristic. These assumptions are very mild. As previously mentioned, landscapes that have existed for a long time should at least be approximately stationary. To the authors' knowledge, all colonisation functions used in practice are Lipschitz continuous.

This theorem states that the process $ ((X^{n}_{i,t},\theta_{i,t},z_{i}), t \geq 0 ) $ converges to a Markov chain and is asymptotically (as $ n \rightarrow \infty$) independent of the rest of the metapopulation. It is possible to extend this result to
the case of a finite collection of habitat patches so that the local population at each patch is independent of the local populations of all others in this collection. Results of this type are sometimes referred to as `propagation of chaos' \citep[for example][Proposition 4.3]{Leonard:90}.

The following two results provide the recursion (\ref{Thm2:Eq1}) - (\ref{Thm2:Eq2}) with a natural interpretation.

\begin{theorem} \label{thmAsym2}
Let $ ((X_{i,t},\theta_{i,t},z_{i}), t \geq 0 ) $ be the Markov chain defined by (\ref{Thm2:Eq0}) - (\ref{Thm2:Eq2}). Then  
\[ 
\mathbb{P} \left(X_{i,t} = 1 \mid \theta_{i,t} = \theta, z_{i} =z \right) = q_{t}(\theta,z)
\]
for all $ t \geq 0 $.
\end{theorem}
Theorem \ref{thmAsym2} shows that $ q_{t}(\theta,z) $ is the probability of a patch located at $ z $ with characteristic $ \theta $ being occupied at time $ t $ in the large metapopulation. 

\begin{theorem} \label{thmAsym3}
Let $ h $ be a continuous function on $ \Theta \times \Omega $. Under the assumptions of Theorem \ref{thmAsym1}, 
\begin{equation}
\lim_{n\rightarrow\infty} n^{-1} \sum_{i=1}^{n} X_{i,t} h(\theta_{i,t},z_{i}) \stackrel{p}{\rightarrow} \int_\om \left( \int_\th  h(\theta,z) q_{t}(\theta,z) \pi(d\theta) \right)\zeta(z) dz . \label{Thm3:Eq1}
\end{equation}
\end{theorem}
For a finite state space, the right hand side of (\ref{Thm3:Eq1}) is $ \int_\om \left(\sum_{j} h(j,z) q_{t}(j,z) \pi_{j} \right) \zeta(z) dz $. If $ a(\theta) $ is a continuous function, then we can apply Theorem \ref{thmAsym3} to expression (\ref{MD:Con:Eq1}) to see that the connectivity measure of patch $ i $ at time $ t $ is approximated by $ \psi_{t}(z_{i}) $ for $ n $ sufficiently large. Theorem \ref{thmAsym3} can also be used to show that $  \int_\om \int_\th   q_{t}(\theta,z) \pi(d\theta) \zeta(z) dz $ is a good approximation to the proportion of colonised patches in the metapopulation at time $ t $ when $ n $ is large.

\section{Equilibrium of a metapopulation with phase structure}  \label{Sec:Equil}

When the number of patches is finite, the metapopulation described by (\ref{MD:Eq2b}) goes extinct in finite time with probability one. However, we have seen in the previous section that when the number of patches in the metapopulation is large, its temporal trajectory is closely tracked by the recursion (\ref{Thm2:Eq1})-(\ref{Thm2:Eq2}). Therefore, if $ q_{t} $ is bounded away from zero for all $ t \geq 0 $, then the metapopulation may persist for a long time. On the other hand if $ q_{t} $ goes to zero, then the proportion of occupied patches converges to zero and the metapopulation will go extinct quickly. In this section, we provide a criterion to distinguish between these two scenarios.

Persistence criteria have been established for a number of population models. \citet{Chesson:84} and \citet{MG:01} derive persistence criteria for structured metapopulation models, that is where the size of the local population at each patch is modelled, not just its presence or absence. \citet{OH:01} established persistence criteria for a deterministic version of the incidence function model and for the spatially realistic Levins' model. In their criteria, the quantity determining persistence factorises into two parts, one which is dependent on the species' dispersal kernel and the landscape, and the other which is a function of non-spatial species specific parameters. The criterion for the spatially realistic Levins' model has been extended to allow suitable/unsuitable patch dynamics \cite{DFS:05,XFAS:06}. These criteria are closely related to the basic reproduction number in disease modelling \cite{vdDW:02,DHR:10}.

While our previous results hold under rather weak assumptions, to analyse the recursion (\ref{Thm2:Eq1})-(\ref{Thm2:Eq2}) we need to impose some more restrictive assumptions. Two main assumptions are used in the analysis. The first is that the metapopulation has a phase structure. This assumption implies that the colonisation function has form
\begin{equation}
f(x;\theta) = s( \theta) \bar{f}(x;\theta) \label{Sec4:eq1}
\end{equation}
for some function $ \bar{f}(x;\theta) : [0,\infty) \times \Theta \rightarrow [0,1] $. The second assumption is that for each $ \theta \in\Theta,\ \bar{f}(\cdot,\theta) $ is concave. The assumption of a concave colonisation function essentially excludes the possibility of an Allee-like effect in the metapopulation \cite{CBG:08}. Although it precludes the colonisation function used by Hanski \cite{Hanski:94}, it is sufficiently weak to accommodate a wide range of functions including the one used in \cite{HC:01}.

The persistence criterion for model (\ref{MD:Eq2b}) is expressed in terms of two quantities. The first quantity $ r_{S} $ is based on how the focal species reacts to fluctuations in the landscape characteristic and is defined by
\[
r_{S} :=\sum_{m=1}^{\infty}   \mathbb{E}\left\{\bar{f}^{\prime}(0;\theta_{0}) \left[ \prod_{n=0}^{m-1} s(\theta_{n})\right] a(\theta_{m})\right\},
\]
where $ (\theta_{t}, t\geq 0) $ is the Markov chain with transition kernel $ P $ and stationary distribution $ \pi $. The second quantity depends on how the focal species disperses in the landscape and on the distribution of patch locations. It is given by the spectral radius $ r(\mathcal{M}) $ of the bounded linear operator $ \mathcal{M} : C(\Omega) \rightarrow C(\Omega) $ defined by
\[
\mathcal{M} \phi(z) :=  \int_\om D(z,\tilde{z}) \phi(\tilde{z}) \zeta(\tilde{z})d\tilde{z}, \quad \phi \in C(\Omega).
\]
Note that this quantity is independent of the landscape dynamics.

\begin{theorem} \label{theorem3}
Suppose the assumptions listed in Appendices A and C hold. If $ r_{S}\times r(\mathcal{M}) \leq 1 $, then recursion (\ref{Thm1:Eq1})-(\ref{Thm1:Eq2}) has only the trivial fixed point $ q(\theta,z) = 0 $ for all $(\theta,z) \in \Theta\times\Omega$, and this fixed point is globally stable. If $ r_{S} \times r(\mathcal{M}) > 1 $, then recursion (\ref{Thm2:Eq1})-(\ref{Thm2:Eq2}) has a unique non-zero fixed point. Furthermore, if $ q_{0}(\theta,z) > 0 $ for all $
(\theta,z) \in \Theta\times\Omega $, then $ q_{t} $ converges to this non-zero fixed point.
\end{theorem}

The assumptions from Appendix A were explained following Theorem \ref{thmAsym1} in the context of a finite state space for the landscape dynamics. The two main assumptions from Appendix C have already been mentioned at the beginning of this section, namely that the colonisation function has the form (\ref{Sec4:eq1}) and for each $ \theta \in\Theta,\ \bar{f}(\cdot,\theta) $ is concave. When $ \Theta $ is finite, the two other assumptions simplify. One is that $ s(i) < 1 $ for each $ i \in \Theta $. For the final assumption define $ \Theta_{1} := \{j \in \Theta : s(j)f(x,j) = 0  \mbox{ for all } x\in [0,\infty) \}$. The set $ \Theta_{1} $ comprises those states of the landscape characteristic for which a patch cannot be colonised. The final assumption is that there is an $ i \in \Theta \backslash \Theta_{1} $ and $ j \in \Theta $ such that $ a(j) > 0 $ and $ P_{ij} > 0 $. This means that, given the landscape characteristic was in a state that allowed the patch to be colonised, there is positive probability that at the next time step the landscape characteristic will be in a state which makes a positive contribution to the colonisation of the other patches. 

Theorem \ref{theorem3} is an example of the kind of dichotomy observed in other metapopulation models not displaying an Allee-like effect. Without the assumption that $ \bar{f} $ is concave, the condition $ r_{S}\times r(\mathcal{M}) > 1 $ still implies the existence of a non-zero fixed point, but there may be several non-zero fixed points in this case. Similarly, if $ \bar{f} $ is only locally concave in a neighbourhood of zero and $ r_{S}\times r(\mathcal{M})  < 1 $,  then the extinction fixed point is still locally stable. That is, iterations of the recursion (\ref{Thm1:Eq1})-(\ref{Thm1:Eq2}) will converge to the fixed point $ q(\theta,z) = 0 $ for all $(\theta,z) \in \Theta\times\Omega$ for sufficiently small initial conditions. This follows from the monotonicity property of the recursion (see the proof of Theorem \ref{thm:Conv}). However, a non-zero fixed point may also exist in this case.

Despite the significant differences between the spatially realistic Levins' model with dynamic landscape studied in \cite{DFS:05,XFAS:06} and the model studied here, there are still some important similarities in their persistence criteria. Firstly, the quantity determining persistence factorises into the product of two terms; one determined by the dispersal kernel of the species and the location of habitat patches, and the other by the species reaction to landscape dynamics and other non-spatial factors. If we suppose that all patches have constant area, so $ a(\theta) = \bar{a} $ for all $ \theta \in \Theta $, then the factor depending on the species' dispersal kernel is the leading eigenvalue of the matrix $ M $ with elements $ M_{ij} := D(z_{i},z_{j}) $. When the patch locations are a sample of independent random variables from the probability density function $ \zeta $, the leading eigenvalue of $ M $ can be shown, under certain conditions, to converge to the spectral radius of $ \mathcal{M} $ as $ n \rightarrow \infty $. On restricting the model so that $ \bar{f}(x;\theta) = \bar{f}(x) $, further similarities appear. We can now write
\[
r_{S} = \bar{a} \bar{f}^{\prime}(0)  \sum_{m=0}^{\infty}   \mathbb{E}\left(\prod_{n=0}^{m-1}s(\theta_{n})\right).
\]
Noting that $ \mathbb{E}(\prod_{n=0}^{m-1}s(\theta_{n})) $ is the probability that a local population survives at least $ m $ extinction phases, we see that the quantity $ \sum_{m=0}^{\infty}  \mathbb{E}(\prod_{n=0}^{m-1}s(\theta_{n})) $ is the expected life span of a local population. For the spatially realistic Levins' model with dynamic landscape, the quantity determining persistence has a factor $ (e + \beta)^{-1} $, where $ e $ is the extinction rate and $ \beta $ is the rate of patch destruction. Therefore, the factor $ (e + \beta)^{-1} $ gives the expected life span of the local population. These similarities suggest a certain amount of robustness of the conclusions to the particular modelling choices made. 

Restricting the model to $ \bar{f}(x;\theta) = \bar{f}(x) $ also enables us to have a clearer understanding of how the landscape dynamics affect the equilibrium level of the metapopulation.

\begin{theorem} \label{thm:LandDist}
Suppose $ \bar{f}(x;\theta) = \bar{f}(x) $. Let $ q^{\ast} $ be a fixed point of recursion (\ref{Thm2:Eq1})-(\ref{Thm2:Eq2}). The quantity $ \int_\th a(\theta) q^{\ast}(\theta,z) \pi(d\theta) $ depends on $ (\pi, P) $ only through the sequence 
\begin{equation} 
\mathbb{E}\left(a(\theta_{m+1})\prod_{n=1}^{m}s(\theta_{n})\right), \quad m \geq 1.  \label{LandDist:Eq1}
\end{equation}
\end{theorem}

Theorem \ref{thm:LandDist} states that, when the metapopulation is in equilibrium, the expected area taken up by colonised patches, $ \int_\th a(\theta) q^{\ast}(\theta,z) \pi(d\theta) $,  depends on the landscape dynamics only through expected future contributions to the connectivity measure of a colonised patch during the local population's life span.  This future contribution to the connectivity measure may be interpreted as the number of propagules produced. If all patches have constant area, then the sequence (\ref{LandDist:Eq1}) is a multiple of the tail probabilities of the life span of the local population. In that case, the probability that a patch at location $ z $ is colonised at equilibrium depends on the landscape dynamics only through the distribution of the local population's life span.

\begin{theorem} \label{thm:SO}
Suppose $ \bar{f}(x;\theta) = \bar{f}(x) $.  Let $ (\theta_{t}, t\geq0 ) $ and $ (\tilde{\theta}_{t}, t \geq 0 ) $ denoted the Markov chains generated by  $ (\pi,P) $ and $ (\tilde{\pi},\tilde{P}) $, respectively,  and let $ q^{\ast} $ and $ \tilde{q}^{\ast} $ be the fixed points of the respective recursions (\ref{Thm1:Eq1})-(\ref{Thm1:Eq2}).  If 
\begin{equation}
\mathbb{E}\left(a(\tilde{\theta}_{m+1})\prod_{n=1}^{m}s(\tilde{\theta}_{n})\right) \leq \mathbb{E}\left(a(\theta_{m+1})\prod_{n=1}^{m}s(\theta_{n})\right), \label{Thm4:Eq1}
\end{equation}
for all $ m \geq 1 $, then
\begin{equation}
\int_\th a(\tilde{\theta}) \tilde{q}^{\ast} (\theta,z) \tilde{\pi}(d\theta) \leq   \int_\th a(\theta) q^{\ast}(\theta,z) \pi(d\theta),  \label{Thm4:Eq2}
\end{equation}
for all $ z \in \Omega $.
\end{theorem}

Theorem \ref{thm:SO} seems intuitively obvious; landscape dynamics which yield greater expected contributions to the connectivity measure result in greater expected area taken up by the colonised patches in equilibrium. What is important is that inequality (\ref{Thm4:Eq1}) must hold for the entire sequence so the expected contribution must be greater at all future times during the local population's life span. If all patches have constant area, then inequality (\ref{Thm4:Eq1}) reduces to stating that the  life span of the local population is larger, in the usual stochastic ordering \cite[Section 1.A]{SS:07}, under $ (\pi,P) $ than under $ (\tilde{\pi},\tilde{P}) $. 

Our final result identifies the Markov chain which makes the life span of the local population the longest for a given stationary distribution $ \pi $.

\begin{corollary} \label{thm:Max}
If the Markov chain $ (\theta_{t}, t\geq 0 ) $ is stationary, then 
\begin{equation}
\mathbb{E}\left(\prod_{n=0}^{m}s(\theta_{n})\right) \leq \mathbb{E}(s(\theta_{0})^{m+1}). \label{Max:Eq1}
\end{equation}
\end{corollary}

This implies that among Markov chains with the same stationary distribution $ \pi $, the one which makes the life span of the local population the greatest is the one for which $ \theta_{0} = \theta_{t} $ for all $ t \geq 0 $, that is the static landscape. This result with Theorem \ref{thm:SO} implies that a static landscape maximises the probability of a patch being occupied when the patch areas are constant. We have not been able to show a similar result for the sequence (\ref{LandDist:Eq1}). The landscape dynamics which maximises the sequence (\ref{LandDist:Eq1}) would seem to be affected by how patch areas and local survival probabilities depend on the landscape dynamics. This is to be expected since the model incorporates the possibility of pulsed dispersal.

To conclude this section, we perform some simulations comparing the proportion of time a patch in the metapopulation is occupied with the limiting probability of patch occupancy determined by $ \int_\th q^{\ast}(\theta,z) \pi(d\theta) $ where $ q^{\ast}(\theta,z) $ is the fixed point of the recursion (\ref{Thm2:Eq1}) - (\ref{Thm2:Eq2}). All simulations are performed with constant patch areas $ a(\theta) = 10 $ for all $ \theta \in \Theta $, $ \bar{f}(x) = 1-\exp(-x) $ and $ D(z,\tilde{z}) = \exp(-\|z-\tilde{z}\|)$. To facilitate the presentation, we assume a one dimensional landscape. The patch locations are sampled from the uniform distribution on $ [0,10] $.

The survival probabilities $ s_{t} = s(\theta_{t}) $ are modelled by the Markov chain studied in \citet{MB:15}. This Markov chain is defined by 
\begin{equation} \label{DLM:Eq2}
s_{t+1} = \left\{ \begin{array}{ll}
s_{t}(1-L_{t+1}), & \mbox{with probability } p(s_{t}), \\
s_{t} + (1-s_{t}) R_{t+1}, & \mbox{with probability } 1 - p(s_{t}),
\end{array} \right.
\end{equation}
where $ p:[0,1] \rightarrow [0,1] $, and $\{L_{t}\}$ and $ \{R_{t}\} $ are sequences of independent and identically distributed random variables on $[0,1]$ with distributions $ F_{L} $ and $ F_{R} $, respectively. Two sample paths are plotted in Figure~\ref{Fig1} for two choices of $ F_{L}$ and $ F_{R}$ with $ p(x) = 10 (x-0.9) \mathbb{I}(x>0.9) $. Although we do not prove that this process is not reversible, the plotted sample paths strongly suggest that it is not. Specifically,  the process would not look the same it time were reversed since the large downward jumps in the trajectory would appear as large upwards jumps if time were reversed.

This Markov chain can provide a reasonable model of changes in habitat quality due to disturbance followed by a slow restoration as follows. Immediately after a disturbance, the habitat is low quality so the local survival probability is small. As time progresses, the habitat recovers and the local survival probability increases until some maximal level is reached or the habitat is again disturbed. To capture the rapid decrease in the survival probability following disturbance, $ F_{L} $ should have considerable mass near one.  The relatively slow recovery of the habitat means that $ F_{R} $ should have most of its mass near zero. The function $ p $ reflects the probability of disturbance for a given survival probability and might reasonably be assumed to be increasing. 

\begin{figure} 
\includegraphics[width=7.5cm]{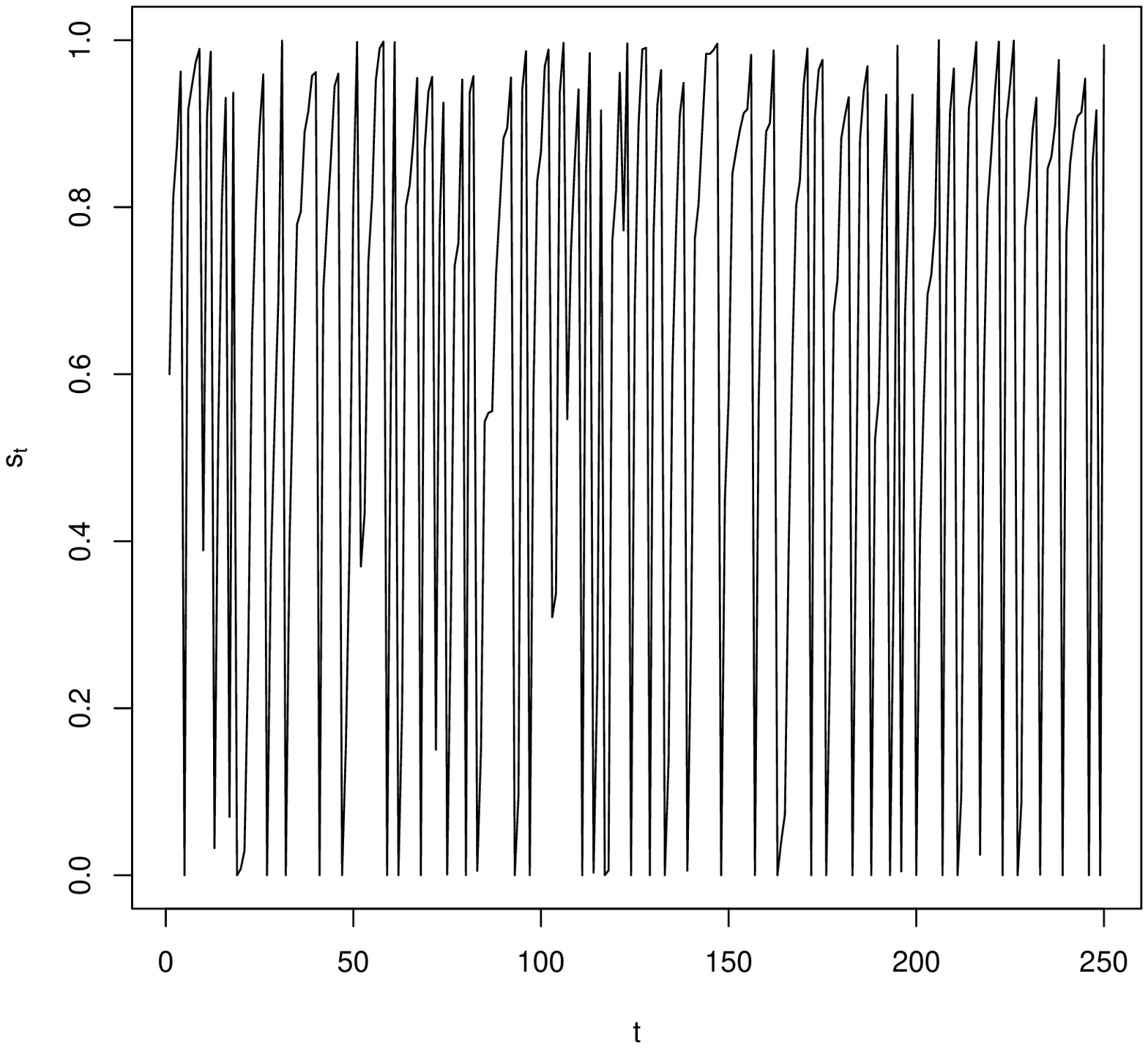}
\includegraphics[width=7.5cm]{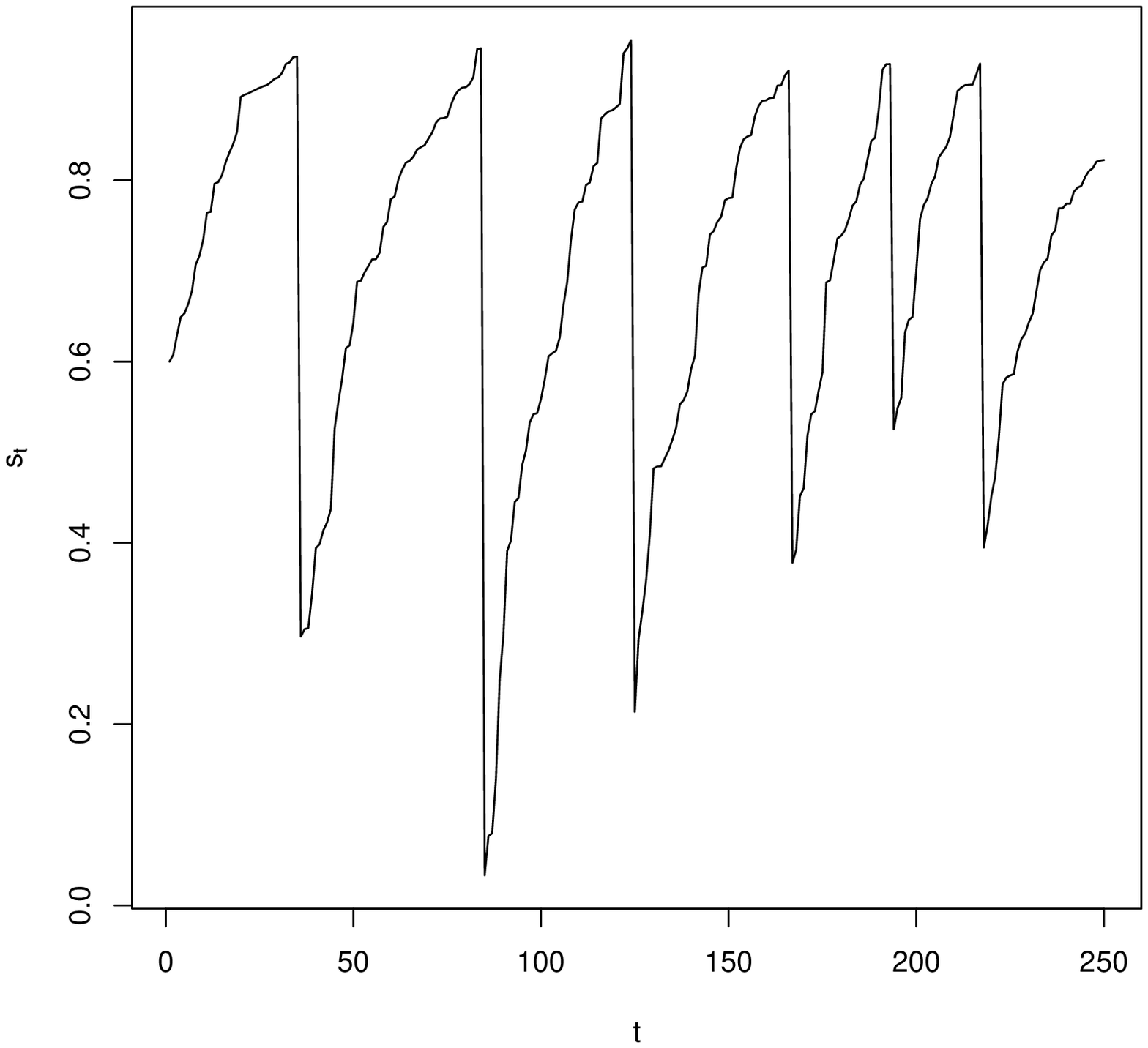}
\caption{Simulated path of survival probabilities. Left: $ L_{t} \sim \mbox{Beta}(1,0.1) $, $ R_{t} \sim \mbox{Beta}(1,1) $. Right: $ L_{t} \sim \mbox{Beta}(1,1) $, $ R_{t} \sim \mbox{Beta}(1,20) $.}\label{Fig1}
\end{figure}

We simulate metapopulations with 50 and 250 habitat patches for $ 10^5 $ time steps with the two survival processes depicted in Figure~\ref{Fig1} and compute the proportion of time each patch is occupied. Treating the resulting time series as stationary, the standard error on the estimated proportions was estimated to be no more than $ 0.003 $ for all simulations. This is compared to the fixed point of the deterministic recursion. Details of how the fixed point is calculated are given in Appendix C. The results are plotted in Figures \ref{Fig2}. As we expect, the fixed point of the deterministic recursion provides a better approximation as the number of patches in the metapopulation increases. It appears that the deterministic recursion
has a greater tendency to over-estimate the proportion of time the patch is occupied than to under-estimate it. Furthermore the deterministic recursion generally provides a better approximation for patches in the center of the metapopulation than those on the periphery. 

\begin{figure}
\includegraphics[width=7.5cm]{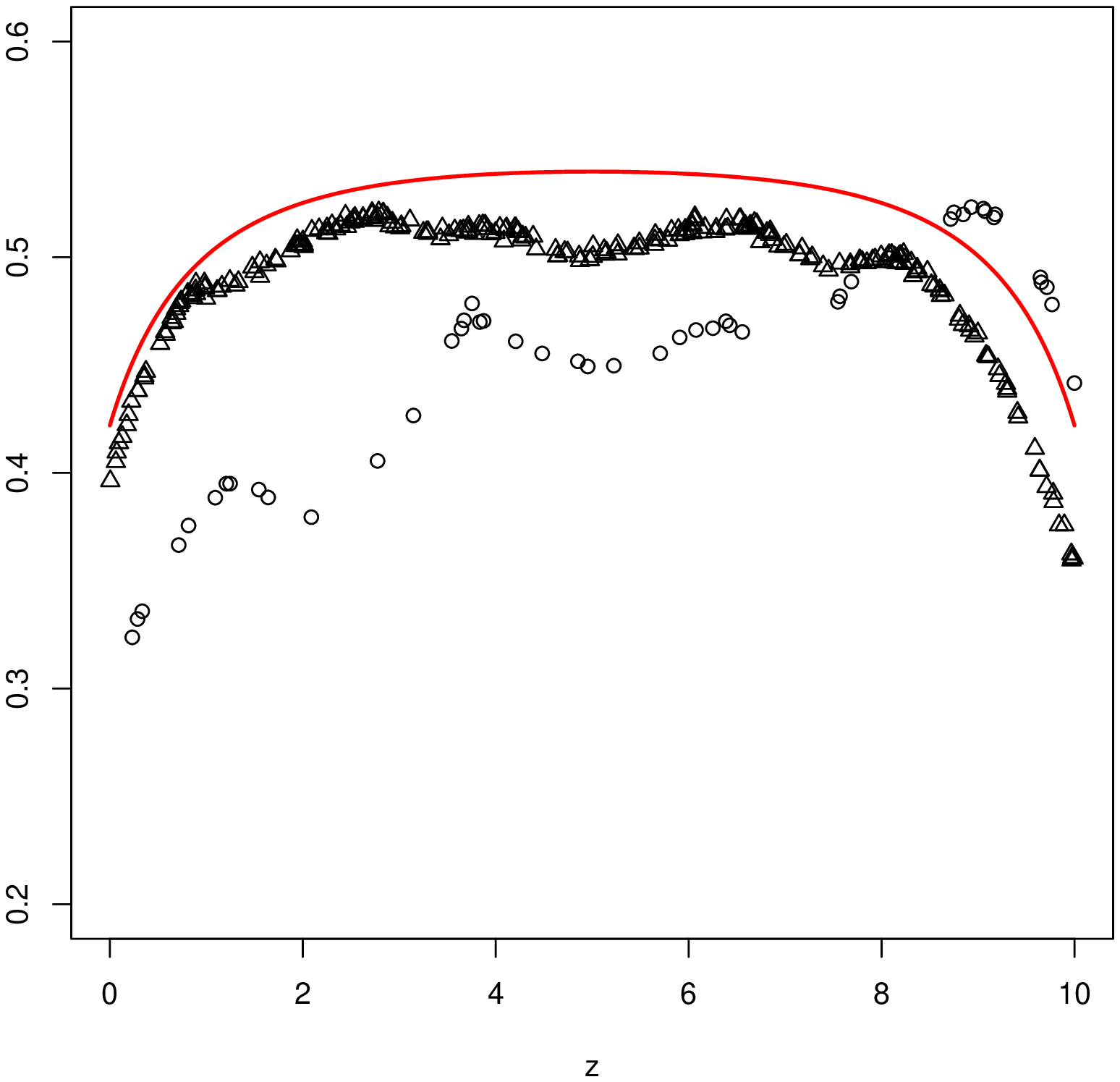}
\includegraphics[width=7.5cm]{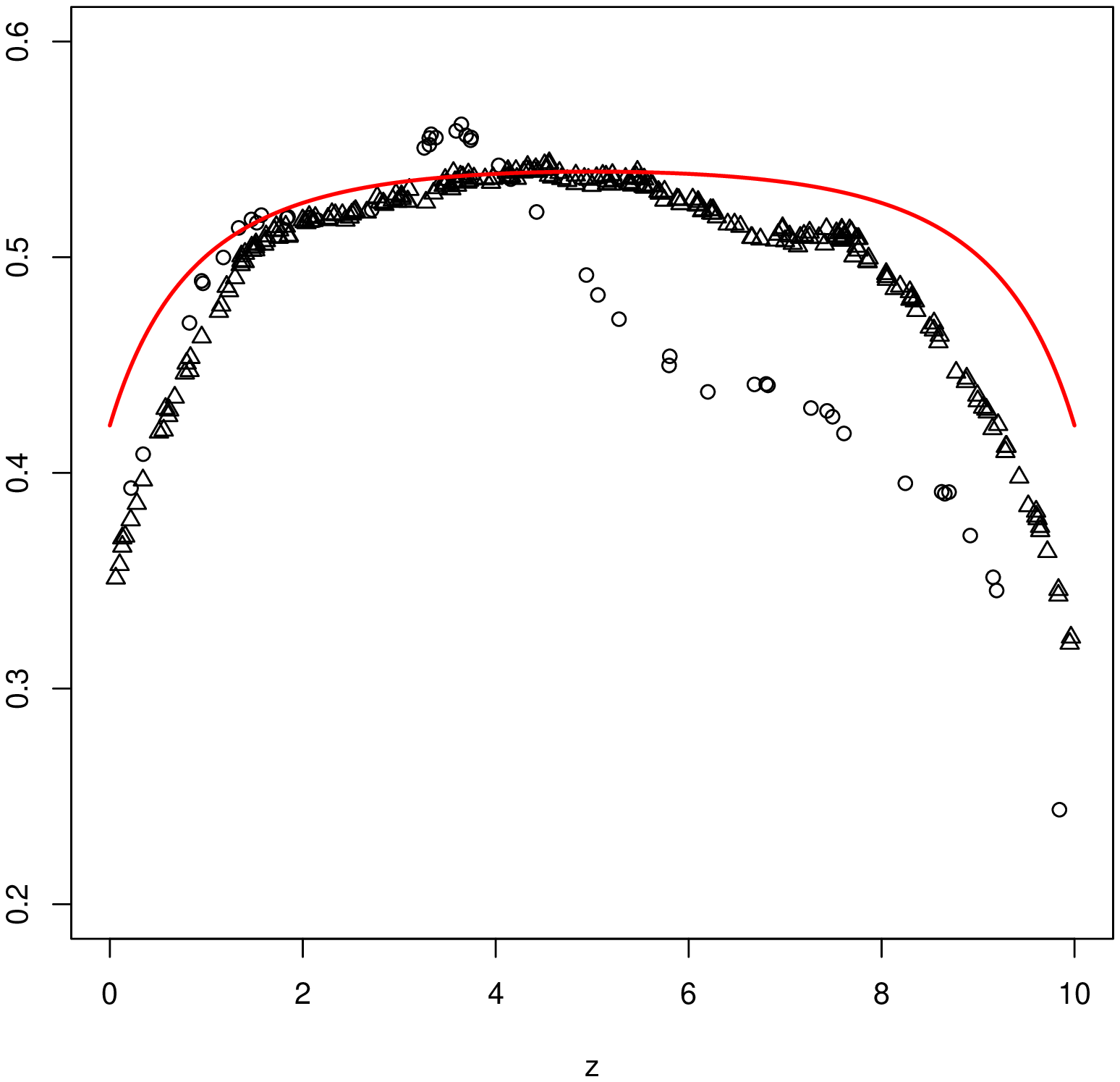}
\includegraphics[width=7.5cm]{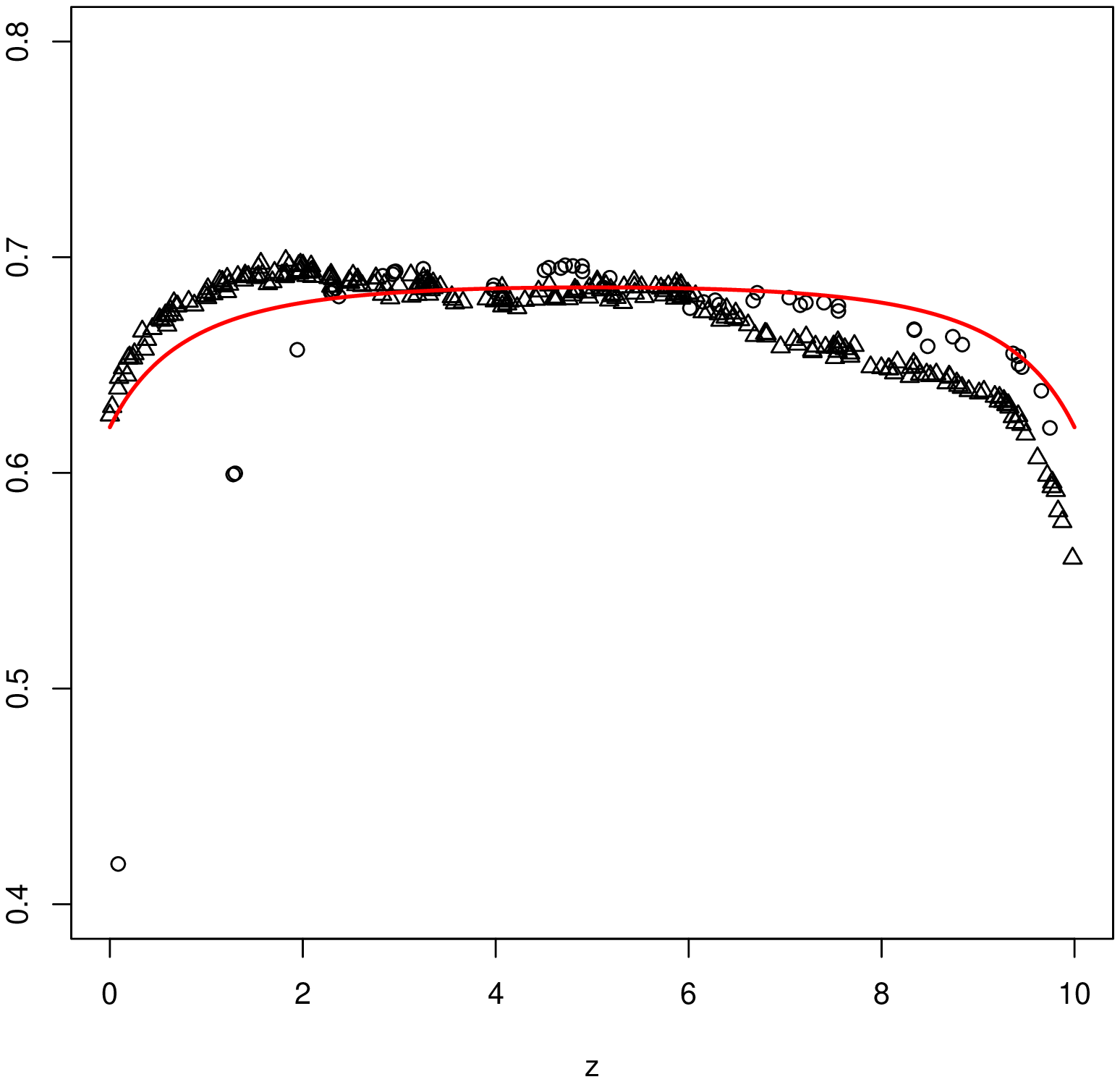}
\includegraphics[width=7.5cm]{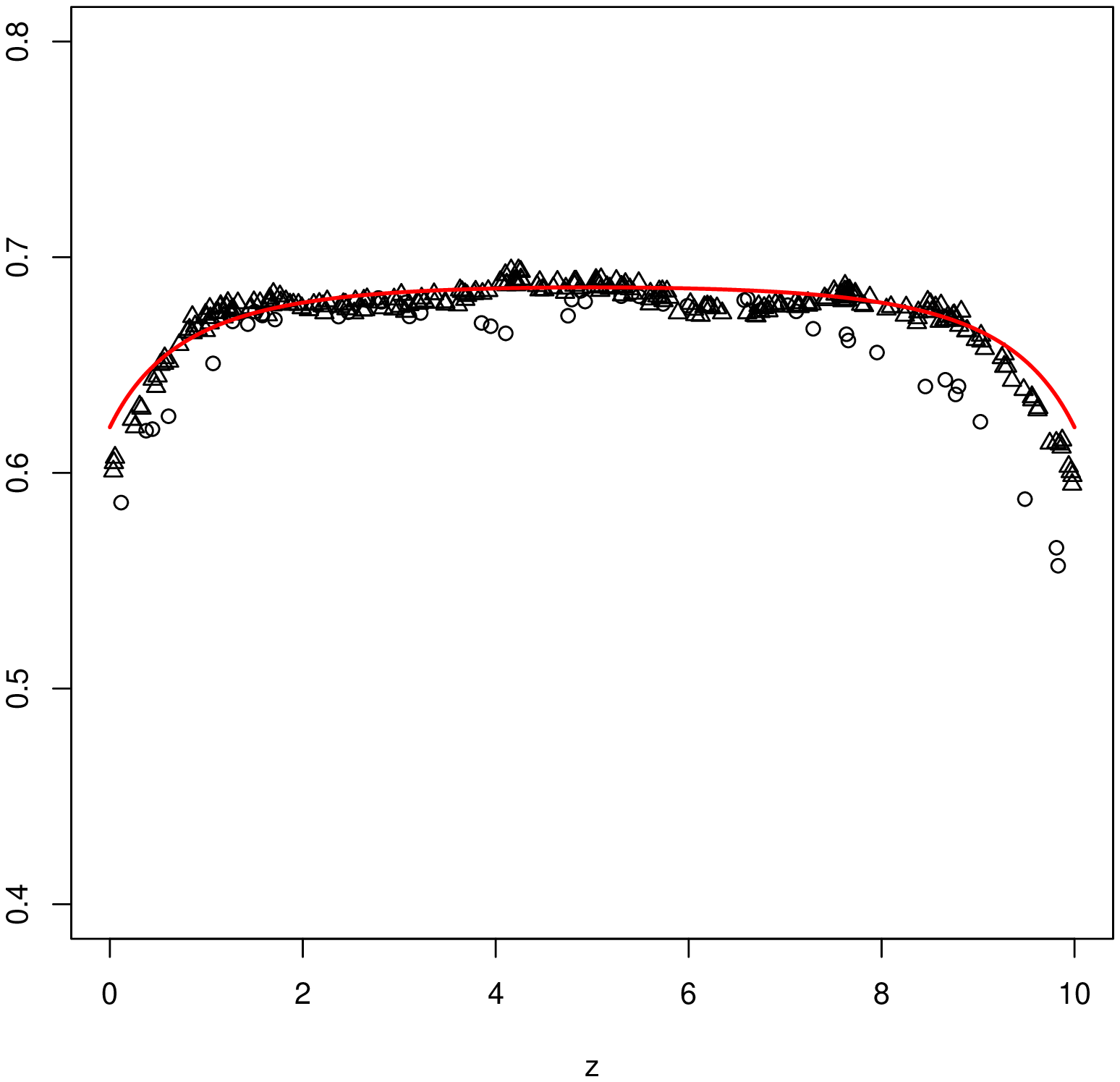}
\caption{The estimated proportion of time each patch is occupied in a
metapopulation with 50 patches ($ \circ $) and 250 patches ($\triangle$)
from simulations. The location of the patch is given by the position of
the points on the $ z $ axis. The line indicates the limiting
probability of patch occupancy. Landscape dynamics for top row took $
L_{t} \sim \mbox{Beta}(1,0.1) $ and $ R_{t} \sim \mbox{Beta}(1,1) $.
Landscape dynamics for bottom row took  $ L_{t} \sim \mbox{Beta}(1,1) $ and $ R_{t} \sim \mbox{Beta}(1,20) $. }\label{Fig2}
\end{figure}

\section{Discussion}

The importance of landscape dynamics to the persistence of metapopulations has been well established in the literature. In contrast to previous mathematical analyses that have employed a suitable/unsuitable classification of habitat patches, we adopted a more general Markovian model for the landscape dynamics to better reflect environmental fluctuations.  It would be of interest to combine the metapopulation model (\ref{MD:Eq2b}) with the Markovian models for succession studied in \cite{Usher:79,Balzter:00,LL:00,BTB:10} among others. The results presented in Sections 3 and 4 would still apply since they were developed for general Markov landscape dynamics. However, using these specific landscape dynamics may reveal a more precise connection with metapopulation survival.

Our analysis yielded similar conclusions to those obtained from models employing the suitable/unsuitable classification of habitat patches. In particular, we note that when certain simplifications are imposed, our persistence criterion (Theorem \ref{theorem3})  has a similar form to the persistence criterion for the spatially realistic Levins' model with dynamic landscape \cite{DFS:05,XFAS:06}. As such, it seems that these conclusions are relatively robust to at least some modification of the model assumptions.  Specifically, we believe the assumptions requiring the model to have the phase structure and for the patch characteristics to be identically distributed across space could be relaxed. Without the phase structure, we expect the existence and uniqueness of the equilibrium to still hold. In fact, much of the current proof could be retained with Lemma \ref{Lem:Mono} being the main difficulty. On the other hand, to prove stability of the equilibrium would need different methods to those currently used. To relax the assumption that patch characteristics are identically distributed across space, we could allow the transition kernel of the patch characteristic, and hence its stationary distribution, to depend on the patch location. This would bring our model closer to the setting described in \cite{DFS:05,XFAS:06} where differences in patch sizes and extinction rates are accommodated. Using the general results from Appendix A, we expect the connectivity measure to still have a deterministic limit and local populations at different patches to be asymptotically independent. However, to obtain a recursion similar to  (\ref{Thm2:Eq0}) - (\ref{Thm2:Eq2}) which we needed to construct the persistence criterion in Theorem \ref{theorem3}, would require a closer investigation of the reversed Markov chain for the patch characteristic.

Of the other assumptions used in the analysis, most are technical assumptions introduced to avoid certain pathological cases. The two assumptions which would significantly impact the results are that the colonisation function is concave and that fluctuations in the landscape are independent between habitat patches. As noted in Section 4, allowing non-concave colonisation functions, such as the one used in \cite{Hanski:94}, would introduce the possibility of a strong Allee-like effect in the metapopulation \cite{CBG:08} so the equilibrium would no longer be globally stable. Independence of patch characteristics at different patches has been used in many other metapopulation models incorporating landscape dynamics \cite{KMVHL:00,DFS:05,Ross:06,WCP:06,XFAS:06,RSAH:15}. However, for certain environmental disturbances such as fires, droughts and floods, the spatial extent can be large compared to the entire habitable area \cite{TBPP:98}, which means the assumption of independence between patches is unlikely to hold. If the metapopulation exhibited some limiting behaviour without the independence assumption, then it would most likely have a very different form.

Dependence between patch characteristics at different patches may not always be obvious. Theorem 3.2 offers the possibility of identifying dependence between patches when the metapopulation is large. For large metapopulations, if patch characteristics are independent at different patches, then the local populations at the two patches are approximately independent. We could estimate the strength of dependence between two patches. Strong dependence would indicate the presence of dependence in the patch characteristics. Unfortunately, testing for independence would not be very useful since the independence of local populations is only asymptotic.

One important way in which our model differs from those using the suitable/unsuitable classification of habitat patches is in the distribution of the local population life span. For static landscapes and dynamics landscapes using the suitable/unsuitable classification, the life span of a local population always has a geometric distribution (discrete time models) or an exponential distribution (continuous time models). We do not have any results characterising the life span distributions permitted by our model, but we expect that almost any distribution is possible by analogy with phase type distributions \cite{OCinneide:90}. 

The effect of landscape dynamics on the equilibrium level of the metapopulation appears quite complicated in general, however for metapopulations with phase structure the role of landscape dynamics is much clearer. The landscape dynamics affect the equilibrium of the metapopulation primarily through the expected future contributions to the connectivity measure of a colonised patch. When the patch area is constant, expected future contributions to the connectivity measure of a colonised patch is determined by the distribution of the local population's life span. Given these results, it is natural to wonder whether the metapopulations with dynamic landscape behave similarly to metapopulations whose local populations have non-geometric/exponential life span distributions, at least in the specialised setting. However, in analysis not reported here we have seen that the equilibrium level in metapopulations with non-geometric life span distributions depends on the life span distribution only through its expectation. This is perhaps not
surprising given this has also been observed in the SIS model with general infectious period distributions \cite{Neal:14}. (Recall the standard SIS model, or stochastic logistic model, has been used a stochastic counterpart to Levins' model \cite{Ovaskainen:01}.)

Finally, it has been observed in metapopulation models with suitable/unsuitable habitat dynamics that metapopulations are more likely to persist and will persist at higher levels of occupancy with static landscapes than with dynamic landscapes. Our analysis shows that this still holds for metapopulations with more general landscape dynamics (Corollary \ref{thm:Max}) if the patch area is constant. However, static landscapes are not necessarily optimal when the patch area is stochastic, leading to similar behaviour to pulsed dispersal \cite{RSAH:15}. The conclusion that static landscapes are optimal may also be false if we move beyond the metapopulation framework and consider species coexistence. Multiple mechanisms have been identified by which landscape dynamics enables the species coexistence \cite{SC:02,RSW:04,MC:09}. These mechanisms are based on different species responding to environmental disturbances in different ways. For example, one species may be better at surviving disturbance, while another may actively try to colonise new areas in response to the disturbance. So although a static landscape may be optimal from the perspective of a single species, it may be to the detriment of other species in the community.

\section{Appendix A --- Proofs for single patch asymptotics}
In this appendix we prove the results of Section \ref{Sec:Asym} in a more general form than stated there. We begin by listing the main assumptions used in our analysis of model (\ref{MD:Eq2b}). For all $ n $:
\begin{itemize}
\item[(A)] The functions $ a(\cdot) $ and $ s(\cdot) $ are continuous on $ \Theta $.  
\item[(B)] Both $ \Theta $ and $ \Omega $ are  compact spaces.
\item[(C)] The colonisation function is continuous in the second argument and satisfies the Lipschitz condition 
\[
\sup_{\theta \in\Theta} \left| f(z;\theta) - f(\tilde{z};\theta) \right| \leq L \left|z-\tilde{z}\right|
\]
for any $ z,\tilde{z} \in \Omega $ and some $ L < \infty $. The colonisation function  is increasing and satisfies $ f(0) = 0 $.
\item[(D)] The function $ D(\cdot,\cdot) $ defines a uniformly bounded and equicontinuous family of functions on $ \Omega $. That is, there exists a finite constant $ \bar{D} $ such that for all $ z,\tilde{z} \in \Omega $, 
$
\left|D(z,\tilde{z})\right| \leq \bar{D},
$
and for every $ \epsilon > 0 $ there exists a $ \delta>0 $ such that for all $ z,\tilde{z} $ with $ \| z - \tilde{z}\| < \delta $ 
$$
\sup_{y\in \Omega} |D(z,y) - D(\tilde{z},y)| < \epsilon.
$$
Furthermore, $ D(z,\tilde{z}) > 0 $ for all $ (z,\tilde{z}) \in \Omega\times\Omega $.
\item[(E)] The transition kernel of the patch characteristic process
satisfies the weak Feller property, that is, for every continuous function $ h $ on $ \Theta $, the function defined by 
$$
Ph(\theta) := \int_\th h(\eta) P(\theta,d\eta), \quad \theta\in\Theta,
$$
is also continuous \citep[Proposition 6.1.1(i)]{MT:96}.
\end{itemize}
We will discuss these assumptions further in Subsection \ref{Subsec:Proof3}. For now we note that Assumptions (A) - (D) are satisfied by typical models. Assumption (E) is a regularity assumption needed when $ \Theta $ is a general state space. It basically requires that the distributions $ P(\theta,\cdot) $ and $ P(\theta^{\prime}, \cdot) $ are close if $ \theta $ and $ \theta^{\prime} $ are close. When $ \Theta $ is a finite state space, Assumption (E) is trivially satisfied.

\subsection{Limiting behaviour of the landscape} \label{Subsec:Land}
We construct random measures $ \sigma_{n,t} $ which summarise the state of the landscape in a metapopulation with $ n $ patches at time $ t $. These measures are purely atomic, placing mass $ n^{-1} $ at the point determined by patch $ i$'s location and its characteristic variable at time~$t$. Let $ C^{+}(\Theta\times\Omega)$ be the space of continuous functions $ h: \Theta \times \Omega \rightarrow [0,\infty) $. By Assumption (B), $ \Omega $ and $ \Theta $ are compact so every function in $ C^{+}(\Theta\times\Omega) $ is bounded. The random measure $ \sigma_{n,t} $ is defined by
$$
\int_\thom h(\theta,z) \sigma_{n,t} (d\theta,dz)  :=  n^{-1} \sum_{i=1}^{n} h(\theta_{i,t},z_{i}), \quad \mbox{for all }  h \in C^{+}(\Theta\times\Omega).
$$
As $ n\rightarrow \infty $, the sequence of random measure $ \sigma_{n,t} $ converges in distribution to $ \sigma_{t} $ if and only if 
\begin{equation}
\int_\thom h(\theta,z) \sigma_{n,t}(d\theta,dz) \stackrel{d}{\rightarrow} \int_\thom h(\theta,z) \sigma_{t}(d\theta,dz), \quad \mbox{for all } h \in C^{+}(\Theta\times\Omega) \label{MeasureConverge}
\end{equation}
\citep[Theorem 16.16]{Kallenberg:02}. Since we are only deal with random measures converging to non-random measures, the convergence in (\ref{MeasureConverge}) can be replaced by convergence in probability. The last of our main assumptions is
\begin{itemize}
\item[(F)] As $ n \rightarrow \infty $, $ \sigma_{n,0} \stackrel{d}{\rightarrow} \sigma_{0} $ for some non--random measure $ \sigma_{0} $. 
\end{itemize}
Although this assumption only concerns the initial variation in the landscape, it implies a similar `law of large numbers' for the landscape at all subsequent times. 

\begin{lemma} \label{lemma1}
Suppose Assumptions~(B), (E) and (F) hold. Then $ \sigma_{n,t} \stackrel{d}{\rightarrow} \sigma_{t} $, where $ \sigma_{t} $ is defined by the recursion
$$
\int_\thom h(\theta,z) \sigma_{t+1} (d\theta,dz) = \int_\thom h(\theta,z) \int_\om P(\eta,d\theta) \sigma_{t}(d\eta,dz), \quad \mbox{for all } h \in C^{+}(\Theta\times\Omega).
$$
\end{lemma}

\begin{proof}
If $ \int_\thom h(\theta,z)\, \sigma_{n,t}(d\theta,dz) \stackrel{d}{\rightarrow} \int_\thom h(\theta,z)\, \sigma_{t}(d\theta,dz) $ for all $ h \in C^{+}(\Theta\times\Omega) $, then $ \sigma_{n,t} \stackrel{d}{\rightarrow} \sigma_{t} $ \citep[Theorem 16.16]{Kallenberg:02}. We use induction on $ t $ to prove weak convergence of the random measures $ \sigma_{n,t} $ to non--random measures $ \sigma_{t} $. By Assumption~(F), $ \sigma_{n,0} \stackrel{d}{\rightarrow} \sigma_{0} $ for some non--random measure $ \sigma_{0}$.  The conditional expectation of $ \int_\thom h(\theta,z)\, \sigma_{n,t+1}(d\theta,dz)  $ given $ (\theta_{t}^{n}, z^{n}) $ is
\begin{align*}
\mathbb{E} \left(\int_\thom h(\theta,z) \sigma_{n,t+1} (d\theta,dz) \mid \theta_{t}^{n}, z^{n} \right) & =  n^{-1} \sum_{i=1}^{n} \int_\om h(\eta,z_{i}) P(\theta_{i,t},d\eta) \\
& = \int_\thom \left\{ \int_\om h(\eta,z) P(\theta,d\eta) \right\} \sigma_{n,t}(d\theta,dz).
\end{align*}
Suppose that $ \sigma_{n,t} \stackrel{d}{\rightarrow} \sigma_{t} $ for some non--random measure $ \sigma_{t}$. If $ \int_\th h(\eta,z) P(\theta,d\eta) $ is in $ C^{+}(\Theta\times\Omega)$, then
\begin{equation}
\lim_{n\rightarrow\infty} \mathbb{E} \left(\int_\thom h(\theta,z) \sigma_{n,t+1} (d\theta,dz) \mid \theta_{t}^{n}, z^{n} \right) = \int_\thom \left\{ \int_\om h(\eta,z) P(\theta,d\eta) \right\} \sigma_{t}(d\eta,dz). \label{ProofLem1:Eq0}
\end{equation}
We now show that $ \int_\th h(\eta,z) P(\theta,d\eta) \in C^{+}(\Theta\times\Omega) $. For any $(\theta',z') \rightarrow (\theta,z) $, 
\begin{align*}
\lim_{(\theta',z') \rightarrow (\theta,z)} \int_\th h(\eta,z')P(\theta',d\eta) &  = \lim_{\theta' \rightarrow \theta}  \int_\th h(\eta,z)P(\theta',d\eta)  \\
& + \lim_{(\theta',z') \rightarrow (\theta,z)}  \int_\th \left[h(\eta,z') - h(\eta,z)\right] P(\theta',d\eta).
\end{align*}
Since $ P $ has the weak Feller property from Assumption~(E), 
$$ 
\lim_{\theta' \rightarrow \theta} \int_\th h(\eta,z)P(\theta',d\eta) =  \int_\th h(\eta,z)P(\theta,d\eta). 
$$
As $ \int_\th P(\theta,d\eta) = 1$, 
$$
\left| \lim_{(\theta',z') \rightarrow (\theta,z)}  \int_\th \left[h(\eta,z') - h(\eta,z)\right] P(\theta',d\eta) \right| \leq \lim_{z'\rightarrow z}\sup_{\theta'\in\Theta} |h(\theta',z') - h(\theta',z) | .
$$
From Assumption (B), $ \Theta \times \Omega $ is compact so the Heine-Cantor Theorem implies that $ h $ is uniformly continuous. Therefore, $ \int_\th h(\eta,z')  P(\theta',d\eta) \rightarrow \int_\th h(\eta,z) P(\theta,d\eta) $ as $ (\theta',z') \rightarrow (\theta,z) $. Hence, $ \int_\th h(\eta,z)P(\theta,d\eta) \in C^{+}(\Theta\times\Omega) $ and equality (\ref{ProofLem1:Eq0}) holds.

The conditional variance of $ \int_\thom h(\theta,z) \sigma_{n,t+1}(d\theta,dz) $ can be bounded by
$$
\mbox{var}\left( \int_\thom h(\theta,z) \sigma_{n,t+1}(d\theta,dz) \mid \theta_{t}^{n}, z^{n} \right) \leq n^{-1}  \sup_{(\theta,z)\in\Theta\times \Omega} |h(\theta,z)|^{2}.
$$
As the conditional variance goes to zero in probability, we can apply a Chebyshev type inequality \cite[Appendix C]{MP:12} to conclude that  
\begin{align}
\int_\thom h(\theta,z) \sigma_{n,t+1}(d\theta,dz)  & \stackrel{p}{\rightarrow}  \int_\thom \left\{ \int_\om h(\eta,z) P(\theta,d\eta) \right\} \sigma_{t}(d\theta,dz). \nonumber \\
& = \int_\thom h(\theta,z) \left\{  \int_\om P(\eta,d\theta)  \sigma_{t}(d\eta,dz) \right\} \label{ProofLem1:Eq1} \\
& = \int_\thom h(\theta,z) \sigma_{t+1}(d\theta,dz). \nonumber
\end{align}
Hence, $ \sigma_{n,t+1} \stackrel{d}{\rightarrow} \sigma_{t+1} $. The recursion for $ \sigma_{t+1} $ is determined by equation (\ref{ProofLem1:Eq1}).
\end{proof}

\subsection{Limiting behaviour of the metapopulation} \label{Subsec:Meta}
Similar to our treatment of the landscape, we construct random measured $ \mu_{n,t} $ which summarise the state of the metapopulation at time $ t $. These measures are defined by
$$
\int_\thom h(\theta,z) \mu_{n,t} (d\theta,dz)  :=  n^{-1} \sum_{i=1}^{n}  X_{i,t}^{n} h(\theta_{i,t},z_{i}), \quad \mbox{for all } h \in C^{+}(\Theta\times\Omega). \label{MR:Eq4}
$$
The measure $ \mu_{n,t}$ has a similar structure to $ \sigma_{n,t}$, but only involves those patches that are occupied at time~$ t $.   Under the stated assumptions, the sequence of random measures $ \{\mu_{n,t}\}_{n=1}^{\infty} $ converges to a deterministic measure as the number of patches tends to infinity.
  
\begin{theorem} \label{theorem1} Suppose that Assumptions (A) -- (F) hold and that $ \mu_{n,0} \stackrel{d}{\rightarrow} \mu_{0} $ for some non--random measure $ \mu_{0} $. Then $ \mu_{n,t} \stackrel{d}{\rightarrow} \mu_{t} $ for all $ t=0,1,\ldots, $ where $ \mu_{t} $ is defined by the recursion
\begin{align}
\lefteqn{\int_\thom h(\theta,z) \mu_{t+1}(d\theta,dz)} \nonumber\\
=  & \int_\thom s(\theta) \left\{ \int_\th h(\eta,z) P(\theta,d\eta) \right\} \mu_{t}(d\theta,dz)  \nonumber \\
& + \int_\thom \left\{ \int_\th h(\eta,z) P(\theta,d\eta) \right\} f\left( \int_\thom a(\tilde{\theta}) D(z,\tilde{z})\mu_{t}(d\tilde{\theta},d\tilde{z}); \theta \right) \sigma_{t}(d\theta,dz) \nonumber\\
&   - \int_\thom \left\{ \int_\th h(\eta,z) P(\theta,d\eta) \right\} f\left( \int_\thom a(\tilde{\theta}) D(z,\tilde{z})\mu_{t}(d\tilde{\theta},d\tilde{z}); \theta \right)  \mu_{t}(d\theta,dz), \label{MR:Eq5}
\end{align}
for all $ h \in C^{+}(\Theta\times\Omega)$.
\end{theorem}

\begin{proof}
The proof follows closely the arguments of the proof of Lemma \ref{lemma1} and the proof of Theorem 3.1 \citep{MP:14}. By assumption $ \mu_{n,0} \stackrel{d}{\rightarrow} \mu_{0} $ for some non--random measure $ \mu_{0}$. Suppose that $ \mu_{n,t} \stackrel{d}{\rightarrow} \mu_{t} $ for some non--random measure $ \mu_{t}$. Then
\begin{align}
&\mathbb{E} \left(\int_\thom h(\theta,z) \, \mu_{n,t+1}(d\theta,dz) \mid
X_{t}^{n}, \theta_{t}^{n}, z^{n} \right) \hspace*{8cm} \nonumber\\
 = & \ n^{-1} \sum_{i=1}^{n} \mathbb{E} \left(h(\theta_{i,t+1},z_{i}) | \theta_{i,t}, z_{i}\right) \mathbb{E}\left(X_{i,t+1}^{n} | X_{t}^{n}, \theta_{t}^{n}, z^{n}\right) \nonumber \\
= & \ \int_\thom s(\theta) \left\{ \int_\th h(\eta,z) P(\theta,d\eta) \right\} \mu_{n,t}(d\theta,dz)  \label{MR:Eq6a} \\
& + \int_\thom \left\{ \int_\th h(\eta,z) P(\theta,d\eta) \right\} f\left( \int_\thom a(\tilde{\theta}) D(z,\tilde{z})\mu_{t}(d\tilde{\theta},d\tilde{z}); \theta \right) \sigma_{n,t}(d\theta,dz)  \label{MR:Eq6b} \\
& - \int_\thom \left\{ \int_\th h(\eta,z) P(\theta,d\eta) \right\} f\left( \int_\thom a(\tilde{\theta}) D(z,\tilde{z})\mu_{t}(d\tilde{\theta},d\tilde{z}); \theta \right) \mu_{n,t}(d\theta,dz)  + \epsilon_{n,t}(h),  \label{MR:Eq6c} 
\end{align}
where
\begin{align*}
|\epsilon_{n,t}(h) |  \leq & 2L \left(\int_\thom\int_\th h(\eta,z) P(\theta,d\eta)
\sigma_{n}(d\theta,dz)  \right)  \\
& \times \sup_{z \in \Omega} \left|\int_\thom a(\tilde{\theta}) D(z,\tilde{z})\mu_{n,t}(d\tilde{\theta},d\tilde{z}) - \int_\thom a(\tilde{\theta}) D(z,\tilde{z})\mu_{t}(d\tilde{\theta},d\tilde{z}) \right|, 
\end{align*}
 as $ f $ is uniformly Lipschitz continuous from Assumption~(C). Ranga Rao \cite[Theorem 3.1]{RR:62} showed that 
$$
\sup_{g \in \mathcal{G}} \left|\int_{\mathcal{X}} g(x) \nu_{n}(dx)-\int_{\mathcal{X}} g(x) \nu(dx) \right| \rightarrow 0,
$$
for a sequence of probability measures $ \nu_{n} $ converging weakly to $ \nu $ and $ \mathcal{G} $ a uniformly bounded and equicontinuous family of functions on $ \mathcal{X}$. Applying a small modification that result and Assumption~(D), it follows that if $ \mu_{n,t} \stackrel{d}{\rightarrow} \mu_{t} $, a non--random measure, then
\[
\sup_{z \in \Omega} \left|\int_\thom a(\tilde{\theta})D(z,\tilde{z})\mu_{n,t}(d\tilde{\theta},d\tilde{z}) - \int_\thom a(\tilde{\theta}) D(z,\tilde{z})\mu_{t}(d\tilde{\theta},d\tilde{z}) \right|  \stackrel{p}{\rightarrow} 0.
\]
To prove convergence of the integrals at (\ref{MR:Eq6a}) - (\ref{MR:Eq6c}), we need both $ s(\theta) \int_\th h(\eta,z)P(\theta,d\eta) $ and $ \int_\th h(\eta,z)P(\theta,d\eta) f(\int_\thom a(\tilde{\theta}) D(z,\tilde{z}) \mu_{t}(d\tilde{\theta},d\tilde{z});\theta) $ to be in $ C^{+}(\Theta\times\Omega) $. From the proof of Lemma \ref{lemma1}, $ \int_\th h(\eta,z)P(\theta,d\eta) \in C^{+}(\Theta\times\Omega) $. With Assumption (A) this implies $ s(\theta) \int_\th h(\eta,z)P(\theta,d\eta) \in C^{+}(\Theta\times\Omega) $. Also $ f(\int_\thom a(\tilde{\theta}) D(z,\tilde{z}) \mu_{t}(d\tilde{\theta},d\tilde{z});\theta) \in C^{+}(\Theta\times\Omega) $ by Assumptions~(C) and (D)  so $ \int_\th h(\eta,z)P(\theta,d\eta) f(\int_\thom a(\tilde{\theta}) D(z,\tilde{z}) \mu_{t}(d\tilde{\theta},d\tilde{z});\theta) \in C^{+}(\Theta\times\Omega) $.
Applying the induction hypothesis, $ \mu_{n,t} \stackrel{d}{\rightarrow} \mu_{t} $ for some non--random measure $ \mu_{t}$. Therefore,
\begin{align*}
&\mathbb{E} \left(\int_\thom h(\theta,z) \, \mu_{n,t+1}(d\theta,dz) \mid X_{t}^{n}, \theta_{t}^{n}, z^{n} \right) \stackrel{p}{\rightarrow}  \int_\thom s(\theta) \left\{ \int_\th h(\eta,z) P(\theta,d\eta) \right\} \mu_{t}(d\theta,dz) \\
& + \int_\thom \left\{ \int_\th h(\eta,z) P(\theta,d\eta) \right\} f\left( \int_\thom a(\tilde{\theta}) D(z,\tilde{z})\mu_{t}(d\tilde{\theta},d\tilde{z}); \theta \right) \sigma_{t}(d\theta,dz) \\
& - \int_\thom \left\{ \int_\th h(\eta,z) P(\theta,d\eta) \right\} f\left( \int_\thom a(\tilde{\theta}) D(z,\tilde{z})\mu_{t}(d\tilde{\theta},d\tilde{z}); \theta \right) \mu_{t}(d\theta,dz).
\end{align*}
The conditional variance of $ \int_\thom h(\theta,z) \mu_{n,t+1}(d\theta,dz) $ can be bounded by $ n^{-1} \sup_{(\theta,z)} |h(\theta,z)|^{2} $. Applying a Chebyshev type inequality \cite[Appendix C]{MP:12}, we conclude that  $ \int_\thom h(\theta,z) \mu_{n,t+1}(d\theta,dz) $ converges to $ \int_\thom h(\theta,z) \mu_{t+1}(d\theta,dz) $ in probability. Hence, $ \mu_{n,t+1} \stackrel{d}{\rightarrow} \mu_{t+1} $ with $ \mu_{t+1} $ determined by the recursion (\ref{MR:Eq5}).
\end{proof}

A consequence of Theorem \ref{theorem1} is that $ (X^{n}_{i,t},\theta_{i,t}) $ converges to a Markov chain with time dependent transition probabilities.
\begin{corollary} \label{corollary1}
Assume the conditions of Theorem \ref{theorem1} hold. If $ X_{i,0}^{n} \stackrel{p}{\rightarrow} X_{i,0} $, then $ X_{i,t}^{n} \stackrel{p}{\rightarrow} X_{i,t} $ for all $ t \geq 0 $, where the transition probability for $ X_{i,t} $ is 
\begin{equation}
\mathbb{P}\left(X_{i,t+1}=1 \mid X_{i,t},\theta_{i,t} = \theta,z_{i}\right)  = s(\theta)X_{i,t} +  f \left( \psi_{t}(z_{i}); \theta\right) \left(1-X_{i,t}\right) 
\end{equation}
and 
\begin{align}
\psi_{t}(z) & =  \int_\thom  a(\tilde{\theta}) D(z,\tilde{z})  \mu_{t}(d\tilde{\theta},d\tilde{z}) \label{Thm2:Eq1b}
\end{align}
\end{corollary}

\begin{proof}
The proof follows the same arguments as the proof of Corollary~1 of~\citet{MP:13}. 
\end{proof}

The following result assumes that the landscape is in equilibrium to simplify the recursion (\ref{MR:Eq5}). 

\begin{theorem} \label{theorem2}
If $ \sigma_{t} = \sigma $ for some measure $ \sigma $ and all $ t \geq 0 $, then $ \mu_{t} $ is absolutely continuous with respect to $ \sigma $ for all $ t \geq 0 $. The Radon-Nikod\'{y}m derivative of $ \mu_{t} $ with respect $ \sigma $, denoted by $ \frac{\partial \mu_{t}}{\partial \sigma} $, is given by the recursion
\begin{align}
\frac{\partial \mu_{t+1}}{\partial \sigma} (\theta,z) & =   \int_\th s(\eta) \frac{\partial \mu_{t}}{\partial \sigma}(\eta,z) P^{\ast}(\theta,d\eta)  \nonumber \\
& + \int_\th f\left( \int_\thom a(\tilde{\theta})D(z,\tilde{z})\frac{\partial \mu_{t}}{\partial \sigma} (\tilde{\theta},\tilde{z}) \sigma(d\tilde{\theta},d\tilde{z}); \eta \right) \left( 1 - \frac{\partial \mu_{t}}{\partial \sigma} (\eta,z) \right)  P^{\ast}(\theta,d\eta) . \label{theorem1:Eq5}
\end{align}
\end{theorem}

\begin{proof}
That $ \mu_{t} $ is absolutely continuous with respect to $ \sigma $ for all $ t \geq 0$ follows from the same arguments as \citet[Lemma 5]{MP:13b}. The recursion for the Radon-Nikod\'{y}m derivative of $ \mu_{t} $ with respect to $ \sigma $ requires the dual kernel (see Appendix B). Applying Corollary \ref{Cor:dual} to the integrals on the right hand side of recursion (\ref{MR:Eq5}), we can express the three terms as
\begin{align*}
\lefteqn{ \int_\thom s(\theta) \left\{ \int_\th h(\eta,z) P(\theta,d\eta) \right\} \mu_{t}(d\theta,dz)} \\  & = \int_\th h(\theta,z) \left\{ \int_\thom s(\eta) \frac{\partial \mu_{t}}{\partial \sigma}(\eta,z) P^{\ast}(\theta,d\eta)\right\} \sigma(d\eta,dz), \\
\lefteqn{\int_\thom \left\{ \int_\th h(\eta,z) P(\theta,d\eta) \right\}  f\left( \int_\thom a(\tilde{\theta})D(z,\tilde{z})\mu_{t}(d\tilde{\theta},d\tilde{z}); \theta \right) \sigma(d\theta,dz)}  \\
& =  \int_\thom  h(s,z) \left\{ \int_\th f\left( \int_\thom a(\tilde{\theta})D(z,\tilde{z})\mu_{t}(d\tilde{\theta},d\tilde{z}); \eta \right)  P^{\ast}(\theta,d\eta) \right\} \sigma (d\theta,dz),
\end{align*}
and
\begin{align*}
\lefteqn{\int_\thom \left\{ \int_\th h(\eta,z) P(\theta,d\eta) \right\}  f\left( \int_\thom a(\tilde{\theta})D(z,\tilde{z})\mu_{t}(d\tilde{\theta},d\tilde{z}); \theta \right) \mu_{t}(d\theta,dz)}  \\
& =  \int_\thom  h(s,z) \left\{ \int_\th f\left( \int_\thom a(\tilde{\theta})D(z,\tilde{z})\mu_{t}(d\tilde{\theta},d\tilde{z}); \eta \right)  \frac{\partial \mu_{t}}{\partial \sigma} (\eta,z) P^{\ast}(\theta,d\eta) \right\} \sigma (d\theta,dz).
\end{align*}
The recursion follows by combining these terms and noting that the Radon-Nikod\'{y}m derivative is uniquely defined up to a $ \sigma$-null set.
\end{proof}

\begin{theorem} \label{corollary2}
Assume that $ (X^{n}_{i,0},\theta_{i,0},z_{i}) \stackrel{d}{=} (X^{n}_{j,0},\theta_{j,0},z_{j}) $ for all $ i,j $. Suppose that 
\[
 \mathbb{P}(X_{i,0}^{n} = 1 \mid \theta_{i,0}=\theta, z_{i} = z) = q_{0}(\theta,z),
\]
for some function $ q_{0} : \Theta\times \Omega \rightarrow [0,1] $. Under the conditions of Corollary \ref{corollary1} and Theorem \ref{theorem2},
\begin{equation}
\mathbb{P} \left(X_{i,t} = 1 \mid \theta_{i,t} = \theta, z_{i} =z \right) = \frac{\partial \mu_{t}}{\partial \sigma} (\theta,z) \label{Cor2:proof:eq0}
\end{equation}
for all $ t \geq 0 $.
\end{theorem}

\begin{proof}
We first show that $ \sigma $ is the distribution of $ (\theta_{i,0},z_{i}) $. For any $ h \in C^{+}(\Theta\times\Omega)$, the sequence of random variables $ \int_\thom h(\theta,z) \sigma_{n,0}(d\theta,dz) $ is uniformly integrable. This sequence of random variables converges in probability as the limiting measure $ \sigma $ is non-random. Therefore,
\begin{equation}
\mathbb{E}\left(\int_\thom h(\theta,z)\sigma_{n,0}(d\theta,dz) \right) \rightarrow \int_\thom h(\theta,z) \sigma(d\theta,dz). \label{Cor2:proof:eq1}
\end{equation}
As  $ (\theta_{i,0},z_{i}) \stackrel{d}{=} (\theta_{j,0},z_{j}) $ for all $ i,j $,
\begin{equation}
\mathbb{E}\left(\int_\thom h(\theta,z)\sigma_{n,0}(d\theta,dz) \right)  =
\mathbb{E}\left(h(\theta_{i,0},z_{i}) \right).  \label{Cor2:proof:eq2}
\end{equation}
Since (\ref{Cor2:proof:eq1}) and (\ref{Cor2:proof:eq2}) hold for all $ h \in C^{+}(\Theta\times\Omega)$, we see that $ \sigma $ is the distribution of $ (\theta_{i,0},z_{i}) $. 

Similarly, we can identify the Radon-Nikod\'{y}m derivative of $ \mu_{0} $ with respect to $ \sigma $ as the function $ q_{0}(\theta,z) $. For any $ h \in C^{+}(\Theta\times\Omega)$, the sequence of random variables $ \int_\thom h(\theta,z) \mu_{n,0}(d\theta,dz) $  convergences in probability and is uniformly integrable. Therefore,
\[
\mathbb{E}\left(\int_\thom h(\theta,z)\mu_{n,0}(d\theta,dz) \right) \rightarrow \int_\thom h(\theta,z) \mu_{0}(d\theta,dz) = \int_\thom h(\theta,z) \frac{\partial \mu_{0}}{\partial \sigma}(\theta,z) \sigma(d\theta,dz).
\]
As  $ (X^{n}_{i,0},\theta_{i,0},z_{i}) \stackrel{d}{=} (X^{n}_{j,0},\theta_{j,0},z_{j}) $ for all $ i,j $,
\begin{align}
\mathbb{E}\left(\int_\thom h(\theta,z)\mu_{n,0}(d\theta,dz) \right) & = \mathbb{E}\left(h(\theta_{i,0},z_{i}) X_{i,0}^{n}\right) \nonumber\\
& = \int_\thom h(\theta,z) \mathbb{P} \left(X_{i,0}^{n} = 1 \mid \theta_{i,0} = \theta , z_{i} = z \right)  \sigma(d\theta,dz) \nonumber\\
& = \int_\thom h(\theta,z) q_{0}(\theta,z)  \sigma(d\theta,dz). \label{Cor2:proof:eq3}
\end{align}
As equation (\ref{Cor2:proof:eq3}) holds for all $ h \in C^{+}(\Theta\times\Omega) $ and the Radon-Nikod\'{y}m derivative is unique, it follows that
\[
\frac{\partial \mu_{0}}{\partial \sigma}(\theta,z) = q_{0}(\theta,z) .
\]
Thus, we have established equality (\ref{Cor2:proof:eq0}) for $ t = 0 $. The proof for $ t > 0 $ proceeds by induction. Although we will always be conditioning on the patch location, this will not be made explicit to simplify the expressions. Let $ \psi_{t}(z) = \int_\thom a(\tilde{\theta})D(z,\tilde{z}) \mu_{t}(d\tilde{\theta},d\tilde{z}) $. Then $ (X_{i,t}, \theta_{i,t}) $ is a Markov chain on $ \{0,1\}\times \Theta $ with transition kernel 
\begin{align*}
\mathbb{P} \left( X_{i,t+1} = 1, \theta_{t+1} \in A \mid X_{i,t} = x, \theta_{i,t} =\theta \right) & =  \left(s(\theta)x + f(\psi_{t}(z);\theta)(1-x)\right) \int_{A} P(\theta, d\eta) \\
\mathbb{P} \left( X_{i,t+1} = 0, \theta_{t+1} \in A \mid X_{i,t} = x, \theta_{i,t} =\theta \right) & =  \left((1-s(\theta))x + (1-f(\psi_{t}(z);\theta))(1-x)\right) \int_{A} P(\theta, d\eta),
\end{align*}
for any measurable set $ A \subset \Theta$. To compute  $ \mathbb{P}(X_{i,t} = 1 \mid \theta_{i,t} = \theta ) $, note that
\begin{align}
&\mathbb{P}(X_{i,t+1} = 1, \theta_{i,t+1} \in A )  \\
= & \int_\thom \mathbb{P}\left( X_{i,t+1}=1, \theta_{i,t+1} \in A \mid X_{i,t} =1, \theta_{i,t} = \theta\right) \mathbb{P} \left(X_{i,t} =1 \mid \theta_{i,t} = \theta \right) \sigma(d\theta,dz)  \nonumber\\
  &  + \int_\thom \mathbb{P}\left( X_{i,t+1}=1, \theta_{i,t+1} \in A \mid X_{i,t} = 0, \theta_{i,t} = \theta\right) \mathbb{P} \left(X_{i,t} = 0 \mid \theta_{i,t} = \theta \right) \sigma(d\theta,dz) \nonumber\\
 = &  \int_\thom  \left\{\int_{A} P(\theta, d\eta)\right\} s(\theta) \mathbb{P} \left(X_{i,t} =1 \mid \theta_{i,t} = \theta \right) \sigma(d\theta,dz)  \nonumber\\
  &  + \int_\thom \left\{\int_{A} P(\theta, d\eta)\right\} f(\psi_{t}(z);\theta) \left( 1- \mathbb{P} \left(X_{i,t} = 1 \mid \theta_{i,t} = \theta \right) \right) \sigma(d\theta,dz). \label{CE:Eq1}
\end{align}
Applying Corollary \ref{Cor:dual} to the integrals in (\ref{CE:Eq1}) gives
\begin{align}
\lefteqn{\int_\thom  \int_\th \bone(\eta \in A) P(\theta, d\eta) s(\theta) \mathbb{P} \left(X_{i,t} =1 \mid \theta_{i,t} = \theta \right) \sigma(d\theta,dz)}  \nonumber\\
  = &  \int_\thom  \left\{\int_\th P^{\ast}(\theta,d\eta) s(\eta) \mathbb{P}\left(X_{i,t} =1 \mid \theta_{i,t} = \eta\right) \right\} \bone(\theta \in A)  \sigma (d\theta,dz) \label{CE:Eq2}
\end{align}
and
\begin{align}
\lefteqn{\int_\thom \int_\th \bone(\eta \in A ) P(\theta, d\eta) f(\psi_{t}(z);\theta)  \left( 1- \mathbb{P} \left(X_{i,t} = 1 \mid \theta_{i,t} = \theta \right) \right) \sigma(d\theta,dz)}  \nonumber\\
  = & \int_\thom  \left\{\in_\th P^{\ast}(\theta,d\eta) f(\psi_{t}(z);\eta)   \left( 1 - \mathbb{P}\left(X_{i,t} =1 \mid \theta_{i,t} = \eta \right)\right) \right\} \bone(\theta \in A)  \sigma (d\theta,dz). \label{CE:Eq3}
\end{align}
Substituting (\ref{CE:Eq2}) and (\ref{CE:Eq3}) into equation (\ref{CE:Eq1}) yields
\begin{align*}
& \mathbb{P}(X_{i,t+1} = 1, \theta_{i,t+1} \in A ) \\
& =  \int_\thom  \left\{\int_\th P^{\ast}(\theta,d\eta) s(\eta) \mathbb{P}\left(X_{i,t} =1 \mid \theta_{i,t} = \eta \right) \right\} \bone(\theta \in A)  \sigma (d\theta,dz) \\
&  +  \int_\thom  \left\{ \int_\th   f(\psi_{t}(z);\eta)  P^{\ast}(\theta,d\eta) s(\eta) \left( 1 - \mathbb{P}\left(X_{i,t} =1 \mid \theta_{i,t} = \eta\right)\right) \right\} \bone(\theta \in A)  \sigma (d\theta,dz) .
\end{align*}
As the Radon-Nikod\'{y}m derivative is unique up to a $ \sigma $-null set,
\begin{align*}
\lefteqn{\mathbb{P} \left(X_{i,t+1} = 1 \mid \theta_{i,t+1} = \theta \right) } \nonumber\\
 = &  \int_\th s(\eta)  \mathbb{P}\left(X_{i,t} =1 \mid \theta_{i,t} = \eta\right)   P^{\ast}(\theta,d\eta)  + \int_\th \ f(\psi_{t}(z);\eta)  \left( 1 -  \mathbb{P}\left(X_{i,t} =1 \mid \theta_{i,t} = r\right) \right) P^{\ast}(\theta,d\eta).  
\end{align*}
Comparing with (\ref{theorem1:Eq5}), we see that if $$ \mathbb{P} \left(X_{i,t} = 1 \mid \theta_{i,t} = \theta \right) = \frac{\partial \mu_{t}}{\partial \sigma} (\theta,z),$$ then  $$ \mathbb{P} \left(X_{i,t+1} = 1 \mid \theta_{i,t+1} = \theta \right) = \frac{\partial \mu_{t+1}}{\partial \sigma} (\theta,z) ,$$ for all $ t \geq 0 $.
\end{proof}

\subsection{Proof of results from Section 3} \label{Subsec:Proof3}

\begin{proof}[Proof of Lemma \ref{lemma2}] \label{Proof:lemma2}

For any $ h \in C(\Theta\times\Omega) $, $\mathbb{E}\left(h(\theta_{i,t},z_{i})\right) = \mathbb{E} \left(\mathbb{E}\left( h(\theta_{i,t},z_{i})\mid \theta_{i,0},z_{i}\right) \right) $. Let $ P^{t} $ be the $ t $-step transition kernel of the Markov chain for patch characteristic. As this Markov chain is positive Harris and aperiodic, it has a unique invariant measure $ \pi $ and, by \citet[Theorem 13.3.3]{MT:96},
$$
\mathbb{E} \left( h(\theta_{i,t},z_{i}) \mid \theta_{i,0} = \theta, z_{i} = z \right) = 
\int_\th h(\eta,z) P^{t}(\theta,d\eta) \rightarrow \int_\th h(\eta,z) \pi(d\eta),
$$
for every $ (\theta, z) $ as $ t \rightarrow \infty $. By the Dominated Convergence Theorem,
\begin{align}
\lim_{t\rightarrow\infty} \mathbb{E}\left(h(\theta_{i,t},z_{i})\right) = \mathbb{E} \left(\int_\th h(\eta,z_{i}) \pi(d\eta) \right) = \int_\om \int_\th h(\eta,z) \pi(d\eta) \zeta(z) dz. \label{Proof3:Eq1}
\end{align}
As (\ref{Proof3:Eq1}) holds for all $ h \in C(\Theta\times\Omega) $, this proves the limit (\ref{lemma2:Eq1})
\end{proof}

The results given in Section \ref{Sec:Asym} follow as special cases of the results in this appendix. The assumptions required for these results to hold are rather weak. Assumption (B) requires $ \Omega $ to be compact, that is a closed and bounded subset of $ \mathbb{R}^{d} $. The condition imposed by Assumption (B) on $ \Theta $ is trivially satisfied when $ \Theta $ is a finite set, so too Assumptions (A) and (E).  Assumption (C) will be satisfied when $ \Theta $ is a finite set if for each $ \theta \in \Theta $, the colonisation function $ f(\cdot,\theta) $ is Lipschitz. This is a fairly weak requirement since in most cases the colonisation function is taken to be smooth. Even for general state spaces, Assumptions (A), (B) and (E) are not very restrictive. For example, if the Markov chain for the patch characteristic can be expressed as $ \theta_{t} = F(\theta_{t-1},W_{t}) $ with the $ W_{t} $ independent and identically distributed random variables taking values in $ \mathbb{R}^{m} $ and $ F :\Theta \times \mathbb{R}^{m} \rightarrow \Theta $ a smooth function, then Assumption (E) is satisfied. Similar to the finite state space case, Assumption (C) will be satisfied for a general state space if for each $ \theta \in \Theta $, the colonisation function $ f(\cdot,\theta) $ is Lipschitz and the Lipschitz constant is bound on $ \Theta $. Finally, Assumption (D) is satisfied by the typical choice $ D(z,z^{\prime}) = \exp(-\alpha\|z- z^{\prime}\|) $, but range limited dispersion kernels such as $ D(z,z^{\prime}) \propto \bone(\|z-z^{\prime}\| \leq r) $ will not satisfy the lower bound imposed in Assumption (D).

Assume now that the $ (X_{i,0}^{n},\theta_{i,0},z_{i}) $ are independent and identically distributed random variables. By  the law of large numbers and Lemma \ref{lemma2},
$$
\int_\thom h(\theta,z) \sigma_{n,0}(d\theta,dz) \stackrel{d}{\rightarrow} \int_\om \int_\th
h(\theta,z) \pi(d\theta) \zeta(z)dz, \quad \mbox{for all } h \in
C^{+}(\Theta\times\Omega).
$$
Therefore, $ \sigma_{n,0} $ converges in distribution to the non-random measure $ \pi \times \zeta $ and Assumption (F) holds. Furthermore, $ \sigma_{t} = \pi \times \zeta$ as $ (\theta_{i,t},z_{i}) \stackrel{d}{=} (\theta_{i,0},z_{i}) $ for all $ i $. The law of large numbers can also be used to show $ \mu_{n,0} \stackrel{d}{\rightarrow} \mu_{0} $. As in the proof of Theorem \ref{corollary2}, if equation (\ref{thmAsym1:eq1}) holds, then $ q_{0} $ is the Random-Nikod\'{y}m derivative of $ \mu_{0} $ with respect to $ \sigma $. Theorem \ref{thmAsym1} now follows from Corollary \ref{corollary1} and Theorem \ref{theorem2}. Theorem \ref{thmAsym2} follows from Theorem \ref{corollary2} and Theorem \ref{thmAsym3} follows from Theorem \ref{theorem1}.

\section{Appendix B --- Dual process construction}

The dual kernel has been used by various authors studying Markov chains and processes \citep[see][and references therein]{BPZ:95}. As we have been unable to find anything in the literature dealing explicitly with the case of interest here, we state the definition of the dual kernel and some basic results. In the following, $(S, \Sigma) $ denotes a general measurable space. If $ P = \{ P(x,A), x \in S, A \in \Sigma \} $ is such that (i) for each $ A \in \Sigma $, $ P(\cdot,A) $ is a non-negative measurable function on $ S $, and (ii) for each $ x \in S $, $  P(x,\cdot) $ is a measure on $ (S,\Sigma) $ with $ P(x,S) \leq 1 $, then we call $ P $ a sub-transition kernel. If $ P(x,S) = 1 $ for all $ x \in S $, then $ P $ is a transition kernel \citep[pg. 65]{MT:96}

\begin{defn} \label{Def:dual}
Let $P$ be a sub-transition kernel on $(S,\Sigma)$ and let $\pi$ be a $ \sigma$-finite measure on $(S,\Sigma) $. If there exists a sub-transition kernel $P^{\ast} $ such that
\begin{equation}
\int_{A} \pi(dx) P(x,B) = \int_{B} \pi(dx) P^{\ast}(x,A), \label{Def:Eq1}
\end{equation}
for all $A,B \in \Sigma $, then $ P^{\ast} $ is called a \emph{dual of $P$ with respect to} $\pi$. If
\begin{equation}
\int_{A} \pi(dx) P(x,B) = \int_{B} \pi(dx) P(x,A), \label{Def:Eq1a}
\end{equation}
for all $ A,B \in \Sigma $, then $P$ is said to be \emph{reversible with respect to} $\pi$.
\end{defn}

We shall see that if $\pi$ is a subinvariant measure for $P$, then the dual of $P$ with respect to~$\pi$ is determined uniquely $\pi$-almost everywhere, in that, for all $A\in \Sigma$, $P^{\ast}(x,A)$ is the same for $\pi$-almost all $x\in S$. By setting $B$ equal to $S$ in (\ref{Def:Eq1a}), we notice that if $P$ is reversible with respect to $\pi$, then $\pi$ is an invariant measure for $P$. More generally, we have the following.

\begin{theorem} \label{Thm:dual}
Let $ P $ be a sub-transition kernel on $ (S,\Sigma) $ and let $ \pi $ be a $ \sigma$-finite measure on $ (S,\Sigma) $. Then $ \pi $ is a
subinvariant measure for $P$ if and only if there exists a dual $P^{\ast}$ for $P$ with respect to $\pi$. Further, $ \pi $ is an invariant measure for $ P $ if and only if $ P^{\ast} $ is a transition kernel. If $ P^{\ast} $ is dual for $ P $, then $ \pi $ is invariant for $ P^{\ast} $ if and only if $ P $ is a transition kernel.
\end{theorem}

\begin{proof}
Let $P$ be a sub-transition kernel on $(S,\Sigma)$ and let $ \pi $ be a $\sigma$-finite measure on $(S,\Sigma)$. We first show that if $\pi$ 
is a subinvariant measure for $P$, then there exists a sub-transition kernel $ P^{\ast} $ satisfying Definition \ref{Def:dual}. Suppose $\pi
$ is subinvariant for~$P$. For $A \in\Sigma$, define 
$$
\eta_{A}(\cdot) := \int_{A} \pi(dx)P(x,\cdot).
$$
It is a measure on $(S,\Sigma)$ because $P(x,\cdot)$ is a measure on $(S,\Sigma)$. It is also clear that $\eta_{A}$ is absolutely continuous with respect to $\pi$, because if $N\in\Sigma $ is any $\pi$-null set then
$$
\eta_{A}(N)=\int_{A}\pi(dx) P(x,N)\leq \int_{S} \pi(dx) P(x,N) \leq \pi(N) = 0.
$$
So, by the Radon-Nikod\'{y}m theorem, there exists a function $P^{\ast} : S \times \Sigma \rightarrow [0,\infty) $ such that $P^{\ast}(\cdot,A)$  is a $\Sigma$-measurable function, and for all $B\in\Sigma$,
$$
\int_{A} \pi(dx) P(x,B) = \eta_{A}(B) = \int_{B} \pi(dx) P^{\ast}(x,A).
$$
Hence, $P^{\ast}$ is determined uniquely $\pi$-almost everywhere by equation (\ref{Def:Eq1}). It remains to show that, for $\pi$-almost
all $x \in S$, $P^{\ast}(x,\cdot)$ is a measure on $(S,\Sigma)$ with $P^{\ast}(x,S) \leq 1$.

For any $A\in\Sigma $, $P^{\ast}(\cdot,A) $ is the Radon-Nikod\'{y}m derivative of $\eta_{A}$ with respect to $\pi$. As $\eta_{\emptyset}$
is the null measure, $P^{\ast}(x,\emptyset) = 0$ for $\pi$-almost all $ x \in S $. To show that $P^{\ast}(x,\cdot)$ is countably additive, let $ \{B_{k}\} $ be a sequence of pairwise disjoint sets in $ \Sigma$. We want to show that the Radon-Nikod\'{y}m derivative of $ \eta_{\cup_{k} B_{k}} $ with respect to $ \pi $ is $ \sum_{k} P^{\ast}(\cdot,B_{k}) $. For any $ A \in \Sigma $,
\begin{align*}
\eta_{\cup_{k} B_{k}} (A) & = \int_{\cup_{k} B_{k}} \pi(dx) P(x,A)  \\
& = \sum_{k} \int_{B_{k}} \pi(dx) P(x,A) \\
& = \sum_{k} \int_{A} \pi(dx) P^{\ast}(x,B_{k}) \\
& = \int_{A} \pi(dx) \sum_{k} P^{\ast}(x,B_{k}).
\end{align*}
Hence, $ P^{\ast}(x,\cup_{k} B_{k}) = \sum_{k} P^{\ast}(x,B_{k}) $ for $\pi$-almost all $ x \in S $.  Finally, since $\pi$ is subinvariant for $P$, we have, for any $A \in \Sigma $,
$$
\int_{A} \pi(dx) P^{\ast}(x,S) = \int_{S} \pi(dx) P(x,A) \leq \pi(A).
$$
Hence, by the Radon-Nikod\'{y}m Theorem, $P^{\ast}(x,S) \leq 1 $ for $\pi$-almost all $x \in S$.

We now show that if there exists a dual $ P^{\ast} $ for $ P $ with respect to $\pi $, then $ \pi $ is subinvariant. Since $ P^{\ast} $ is
a sub-transition kernel, $ P^{\ast}(x,S) \leq 1 $ for all $ x \in S$. On setting $B$ equal to $S$ in equation (\ref{Def:Eq1}) we see that
\begin{equation}
\int_{S} \pi(dx) P(x,A) 
  = \int_{A} \pi(dx) P^{\ast}(x,S) \leq \int_{A} \pi(dx) = \pi(A),
  \label{Proof:Thm2:eq1}
\end{equation}
that is, $\pi$ is subinvariant for $ P $. This completes the proof of the first part of Theorem~\ref{Thm:dual}.

To prove the second part we note that if $P^{\ast}$ is a transition kernel  then $ P^{\ast}(x,S) = 1 $ for all $x \in S$. In that case, inequality  (\ref{Proof:Thm2:eq1}) becomes equality, and $\pi$ is seen to be invariant. On the other hand, if $\pi$  is invariant for $P$, then
$$
\pi(A) = \int_{S} \pi(dx) P(x,A) = \int_{A} \pi(dx) P^{\ast}(x,S), 
$$
for all $A\in \Sigma $. Therefore, $ P^{\ast}(x,S) = 1 $ for $\pi$-almost all $x \in S$, and $P^{\ast}$ is a transition kernel.  The final part is proved in similar vein.
\end{proof}

\begin{corollary}\label{Cor:dual}
Let $ \phi $ and $ \psi $ be $ \Sigma $-measurable functions. Then, under the conditions of Theorem \ref{Thm:dual}, the dual $ P^{\ast} $ satisfies
$$
\int_{S} \pi(dx) \phi(x) \int_{S} P(x,dy) \psi(y)  = \int_{S} \pi(dx) \psi(x) \int_{S} P^{\ast}(x,dy) \phi(y).
$$
\end{corollary}

\begin{proof}
Suppose the conditions of Theorem \ref{Thm:dual} hold, and $ P^{\ast} $ is a sub-transition kernel that satisfies equation (\ref{Def:Eq1}). Let $ \phi $ and $ \psi $ be the indicator functions $ \phi(x) = \mathbb{I}(x \in A) $ and $ \psi(x) = \mathbb{I}(x \in B) $, where $ A, B $ are $ \Sigma$-measurable sets. Then
\begin{align*}
\int_S \pi(dx) \phi(x) \int_S P(x,dy) \psi(y) & = \int_S \pi(dx)   \mathbb{I}(x \in A)  \int_S P(x,dy)  \mathbb{I}(y \in B) \\
& = \int_{A} \pi(dx)   P(x,B) \\
& = \int_{B} \pi(dx)   P^{\ast}(x,A) \\
& = \int_S \pi(dx) \mathbb{I}(y \in B)  \int_S P^{\ast}(x,dy) \mathbb{I}(y \in A) \\
& = \int_S \pi(dx) \psi(x) \int_S P^{\ast}(x,dy) \phi(y).
\end{align*}
The result holds for indicator functions and can be extended by linearity of the integrals to simple functions $ \phi(x) = \sum_{k} a_{k} \mathbb{I}(x \in A_{k}) $ and $ \psi(x) = \sum_{k} b_{k} \mathbb{I}(x \in B_{k}) $, where $ a_{k}, b_{k} \in \mathbb{R} $ and $ A_{k}, B_{k} $ are $ \Sigma$-measurable sets. Now let $\phi$ and $\psi$  be any $\Sigma$-measurable functions, then we can decompose them as $\phi = \phi^{+} - \phi^{-} $ and $ \psi = \psi^{+} - \psi^{-} $,  where $ \phi^{+}, \phi^{-}, \psi^{+}, \psi^{-} \geq 0 $ are $\Sigma$-measurable functions. Then there exists sequences of non-negative, non-decreasing simple functions $ (\phi^{+}_{n}), (\phi^{-}_{n}), (\psi^{+}_{n}) $ and $ (\psi^{-}_{n})$  such that $ \phi^{+}_{n} \rightarrow \phi^{+}, \phi^{-}_{n} \rightarrow \phi^{-}, \psi^{+}_{n} \rightarrow \psi^{+} $ and $ \psi^{-}_{n} \rightarrow \psi^{-}$, with convergence interpreted pointwise. The result follows by applying the Monotone Convergence Theorem and linearity of integration.
\end{proof}

\section{Appendix C --- Proofs of equilibrium properties} \label{AppendC}

In this appendix we prove the results of Section \ref{Sec:Equil}. As previously noted, our analysis of the recursion (\ref{Thm2:Eq1})-(\ref{Thm2:Eq2}) assumes the model has a phase structure and the form of the colonisation function essentially excludes the Allee effect. We will work with the more general form of the recursion given by equation (\ref{theorem1:Eq5}) which reduces to (\ref{Thm2:Eq1})-(\ref{Thm2:Eq2}) when $ \sigma $ has the product form discussed in Section \ref{Sec:Asym}. We state our assumptions more fully here together with two other technical assumptions.
\begin{itemize}
\item[(G)] Phase structure: $ s(\theta) \geq f(x,\theta) $  for all $ (x,\theta) \in [0,\infty) \times \Theta $.
\item[(H)] No Allee effect: Define $ \Theta_{1} = \{\theta \in \Theta : f(x,\theta) = 0  \mbox{ for all } (x,\theta) \in [0,\infty) \}$. For all $ \theta \in \Theta\backslash \Theta_{1} $, $ f(x,\theta) $ is strictly concave in $ x $.
\item[(I)]  $ \sup_{\theta \in \Theta} s(\theta) < 1 $.
\item[(J)] For every $ \phi \in C^{+}(\Omega) $, 
\[
\sigma\left( \{ (\theta,z) \in \Theta \times\Omega: \phi(z) > 0 \mbox{ and } P^{\ast}(\theta,\Theta_{1}) < 1 \mbox{ and } a(\theta) > 0 \} \right) > 0.
\]
\end{itemize}
The importance of Assumptions (G) and (H) has already been discussed in Section \ref{Sec:Equil}. The only additional point concerning Assumption (H) to make is that the set $ \Theta_{1} $ is interpreted as the set of patch characteristics which are incompatible with the patch being colonised. Assumption (I) implies that there is no patch characteristic that makes the survival of the local population to the next time period certain. Assumption (J) can be simplified when $ \sigma $ is the product measure $ \zeta \times \pi $ since
\begin{align*}
& \sigma\left( \{ (\theta,z) \in \Theta \times\Omega: \phi(z) > 0 \mbox{ and } P^{\ast}(\theta,\Theta_{1}) < 1 \mbox{ and } a(\theta) > 0 \} \right) \\
&= \left(\int_\om \mathbb{I}(\phi(z) > 0) \zeta (z) dz \right) \left( \int_\th \mathbb{I}(P^{\ast}(\theta,\Theta_{1}) < 1 \mbox{ and } a(\theta) > 0 ) \pi(d\theta)\right).
\end{align*}
As $ \Omega $ is the support of $ \zeta $, it follows $ \int_\om \mathbb{I}(\phi(z) > 0) \zeta (z) dz  > 0 $ for any $ \phi \in C^{+}(\Omega) $. The conditions on $ \pi $ implied by Assumption (J)  are more subtle. If $ P^{\ast}(\theta,\Theta_{1}) < 1 $ for all $ \theta \in \Theta $, then Assumption (J) requires that the set of patch characteristics corresponding to patches with positive area has positive $ \pi $ measure.

Our analysis uses the following notation. For any two functions $ \phi $ and $ \psi $ defined on a common domain $ \mathcal{X} $ we write $ \phi \leq \psi $ if $ \phi(x) \leq \psi(x) $ for all $ x \in \mathcal{X} $. Similarly, we write $ \phi < \psi $ if $ \phi \leq \psi $ and $ \phi(x) < \psi(x) $ for some $ x \in \mathcal{X}$. Finally, we write $ \phi \ll \psi $ if $ \phi(x) < \psi(x) $ for all $ x \in \mathcal{X} $. With slight abuse of notation, we let $ 0 $ denote the function that is zero for all $ x \in \mathcal{X} $.

\begin{theorem} \label{thm:CLST}
Suppose Assumptions (A)-(D), (G)-(J) hold. For the recursion (\ref{theorem1:Eq5}) either: (i) $ 0 $ is the unique fixed point in $ C^{+}(\Theta \times \Omega) $ or (ii) there are two fixed points in $ C^{+}(\Theta\times \Omega) $ of which one is $ 0 $ and the other $ q^{\ast} $ satisfies $ 0 \ll q^{\ast} $. 
\end{theorem}

\begin{proof}
For any $ \phi \in C^{+}(\Omega) $, defined the recursion
\begin{equation}
q_{t+1}^{\phi} (\theta,z) = \int_\th s(\eta) q_{t}^{\phi}(\eta,z) P^{\ast}(\theta,d\eta)  + \int_\th f(\phi(z),\eta) (1 - q_{t}^{\phi}(\eta,z)) P^{\ast}(\theta,d\eta). \label{AppC:Eq1}
\end{equation}
Under Assumption (I),  the iterations of (\ref{AppC:Eq1}) converge to a unique fixed point which we denote by $ q^{\phi}_{\infty} $. Let $ (\theta_{t}^{\ast}, t\geq 0) $ be the Markov chain $ (\pi,P^{\ast}) $. We may express $ q^{\phi}_{\infty} $ as
\begin{equation}
q^{\phi}_{\infty}(\theta,z) = \sum_{m=1}^{\infty} \mathbb{E}
\left\{f(\phi(z);\theta_{m}^{\ast}) \prod_{n=1}^{m-1}
\left(s(\theta^{\ast}_{n}) - f(\phi(z);\theta_{n}^{\ast})\right) \mid
\theta_{0}^{\ast} = \theta \right\}. \label{AppC:Eq2}
\end{equation}
Now define the operator $ \mathcal{H} : C^{+}(\Omega) \rightarrow C^{+}(\Omega) $ by
\begin{equation}
\left( \mathcal{H} \phi\right) (z) := \int_\thom a(\tilde{\theta}) D(z,\tilde{z}) q^{\phi}_{\infty}(\tilde{\theta},\tilde{z}) \sigma(d\tilde{\theta},d\tilde{z}). \label{AppC:Eq3}
\end{equation}
Note that if $ \phi $ is a fixed point of $ \mathcal{H} $, then $ q^{\phi}_{\infty} $ is a fixed point of the recursion (\ref{theorem1:Eq5}). Furthermore, any fixed point of (\ref{theorem1:Eq5}) can be used to construct a fixed point of $ \mathcal{H} $. The theorem will be proved if we can show that $ \mathcal{H} $ satisfies the five conditions of the cone limit set trichotomy \citep[Theorem 13]{HS:05}.  

\begin{lemma}[Monotonicity] \label{Lem:Mono}
For any $ \phi,\psi \in C^{+}(\Omega) $, if $ \phi \leq \psi $, then $ \mathcal{H} \phi \leq \mathcal{H} \psi$.
\end{lemma}

\begin{proof} If $ q_{t}^{\phi} \leq q_{t}^{\psi} $, then for all $ \eta \in \Theta $
\begin{align*}
s(\eta) q_{t}^{\phi} (\eta,z) + f(\phi(z),\eta) (1-q_{t}^{\phi}(\eta,z)) & \leq s(\eta) q_{t}^{\phi} (\eta,z) + f(\psi(z),\eta)  (1-q_{t}^{\phi}(\eta,z)) \\
& \leq s(\eta) q_{t}^{\psi} (\eta,z) + f(\psi(z),\eta)  (1-q_{t}^{\psi}(\eta,z)),
\end{align*}
where the first inequality follows as $ f $ is monotone and $ \phi \leq \psi $ and the second follows from Assumption (G) and $ q_{t}^{\phi} \leq q_{t}^{\psi} $. After integrating both sides of the inequality over $ \eta $ with respect to $ P^{\ast}(\theta,d\eta) $, we see $ q^{\phi}_{t+1} \leq q^{\psi}_{t+1} $. We may take $ q_{0}^{\phi} = q_{0}^{\psi} $ so $ q^{\phi}_{t} \leq q^{\psi}_{t} $ for all $ t \geq 0 $. Since the iterations of (\ref{AppC:Eq1}) converge, $ q^{\phi}_{\infty} \leq q^{\psi}_{\infty} $. Substituting into (\ref{AppC:Eq3}), we see that $ \mathcal{H}\phi \leq \mathcal{H}\psi $. 
\end{proof}

\begin{lemma}[Strong sublinearity]
If $ \lambda \in (0,1) $ and $ 0 \ll \phi $, then $ 0 \ll \mathcal{H} (\lambda \phi) - \lambda \mathcal{H} \phi  $.
\end{lemma}

\begin{proof}
From the monotonicity property, $ q^{\lambda\phi}_{\infty}  \leq q^{\phi}_{\infty} $ for all $ \lambda \in (0,1)$. We begin by showing that $ \lambda q^{\phi}_{\infty} \leq q^{\lambda \phi}_{\infty} $.  As $ q^{\phi}_{\infty} $ and $ q^{\lambda \phi}_{\infty} $ are fixed points of (\ref{AppC:Eq1}),
\begin{align*}
 q^{\lambda\phi}_{\infty}(\theta,z) - \lambda q^{\phi}_{\infty}(\theta,z) & = \int_\th s(\eta) \left[ q^{\lambda\phi}_{\infty}(\eta,z) - \lambda q^{\phi}_{\infty}(\eta,z) \right] P^{\ast}(\theta,d\eta) \\
& + \int_\th \left(f(\lambda \phi(z),\eta) (1 - q_{\infty}^{\lambda \phi}(\eta,z)) - \lambda f(\phi(z),\eta)(1 - q_{\infty}^{\phi}(\eta,z)) \right) P^{\ast}(\theta,d\eta) \\
& = \int_\th s(\eta) \left[ q^{\lambda\phi}_{\infty}(\eta,z) - \lambda q^{\phi}_{\infty}(\eta,z) \right] P^{\ast}(\theta,d\eta) \\
& + \int_\th \left(f(\lambda \phi(z),\eta) -  \lambda f(\phi(z),\eta)\right) (1 - q_{\infty}^{\lambda \phi}(\eta,z)) P^{\ast}(\theta,d\eta) \\
&  + \int_\th \lambda f(\phi(z),\eta)(q_{\infty}^{\phi}(\eta,z) -
q_{\infty}^{\lambda \phi}(\eta,z))  P^{\ast}(\theta,d\eta).
\end{align*}
Since $ q^{\lambda \phi}_{\infty} \leq q^{\phi}_{\infty} $ from the monotonicity property,
\begin{align}
 q^{\lambda\phi}_{\infty}(\theta,z) - \lambda q^{\phi}_{\infty}(\theta,z) & \geq \int_\th s(\eta) \left[ q^{\lambda\phi}_{\infty}(\eta,z) - \lambda q^{\phi}_{\infty}(\eta,z) \right] P^{\ast}(\theta,d\eta) \nonumber\\
& + \int_\th \left(f(\lambda \phi(z),\eta) -  \lambda f(\phi(z),\eta)\right)
(1 - q_{\infty}^{\lambda \phi}(\eta,z)) P^{\ast}(\theta,d\eta). \label{AppC:Eq4}
\end{align}
We can bound $ q^{\phi}_{\infty} $ from above as follows. As $ 0 \leq q_{t}^{\phi}(\theta,z) \leq 1 $ for all $ (\theta,z) \in \Theta\times\Omega $, (\ref{AppC:Eq1}) implies
\[
q_{t+1}^{\phi}(\theta,z) \leq \int_\th \left(s(\eta) \vee f(\phi(z),\eta)\right) P^{\ast}(\theta,d\eta).
\]
By Assumption (G), $ q_{t+1}^{\phi}(\theta,z) \leq \sup_{\eta\in\Theta} s(\eta) $ for all $ (\theta,z) \in \Theta \times \Omega $.   Taking $ q_{0}^{\phi} = 0 $ and noting that the iterations of (\ref{AppC:Eq1}) converge to $ q^{\phi}_{\infty} $, it follows $ q_{\infty}^{\phi}(\theta,z) < \sup_{\eta\in\Theta} s(\eta) $ for all $ (\theta,z) \in \Theta \times \Omega $. Iterating inequality (\ref{AppC:Eq4}) and applying Assumption (I), we see
\begin{equation}
 q^{\lambda\phi}_{\infty}(\theta,z) - \lambda q^{\phi}_{\infty}(\theta,z)  \geq \left(1 - \sup_{\eta\in\Theta} s(\eta)  \right) \int_\th \left(f(\lambda \phi(z),\eta) -  \lambda f(\phi(z),\eta)\right)  P^{\ast}(\theta,d\eta) \label{AppC:Eq6}
\end{equation}
so $ \lambda q^{\phi}_{\infty} \leq q^{\lambda \phi}_{\infty} $. 

Now if $ (\mathcal{H} (\lambda \phi))(z) - \lambda (\mathcal{H} \phi)(z) = 0 $ for some $ z \in \Omega $, then 
\begin{equation}
\int_\thom a(\theta) D(z,\tilde{z}) \left[ q_{\infty}^{\lambda\phi}(\theta,\tilde{z}) - \lambda q^{\phi}_{\infty}(\theta,\tilde{z})\right] \sigma(d\theta,d\tilde{z}) = 0. \label{AppC:Eq8}
\end{equation}
Equation (\ref{AppC:Eq8}) and Assumption (D) together imply that $a \cdot [q^{\lambda \phi}_{\infty} - \lambda q^{\phi}_{\infty}] = 0 $, $ \sigma$-almost everywhere.  As $0 \ll \phi $, $ f(\lambda \phi(z),\theta) -  \lambda f(\phi(z),\theta) > 0 $ for all  $ \theta \in \Theta\backslash \Theta_{1} $ and all $ z \in \Omega $.  Using the lower bound (\ref{AppC:Eq6}) and Assumption (J), we see that  $a \cdot [q^{\lambda \phi}_{\infty} - \lambda q^{\phi}_{\infty}]  $ is positive on a set of positive $ \sigma$-measure. Therefore, equation (\ref{AppC:Eq8}) cannot hold for any $ z \in \Omega $. 
\end{proof}

\begin{lemma}[Strong positivity]
If $ 0 < \phi $, then $ 0 \ll \mathcal{H} $.
\end{lemma}

\begin{proof}
If $ (\mathcal{H}\phi)(z) = 0  $ for some $ z \in \Omega $, then
\begin{equation} 
\int_\thom a(\theta) D(z,\tilde{z}) q_{\infty}^{\phi}(\theta,\tilde{z}) \sigma(d\theta,d\tilde{z}) = 0. \label{Lem:SP1}
\end{equation}
By Assumption (D), this implies $ a(\theta)  q_{\infty}^{\phi}(\theta,z) = 0 $ $ \sigma$-almost everywhere. Assumption (G) and equation (\ref{AppC:Eq1}) imply
\begin{equation}
q_{\infty}^{\phi}(\theta,z) \geq \int_\th f(\phi(z),\eta) P^{\ast}(\theta,d\eta).  \label{AppC:Eq5}
\end{equation}
Using the lower bound (\ref{AppC:Eq5}) and Assumption (J), we see that  $a \cdot q^{\phi}_{\infty} $ is positive on a set of positive $ \sigma$-measure. Hence, (\ref{Lem:SP1}) cannot hold for any $ z\in\Omega$.
\end{proof}

\begin{lemma}[Continuity]
The operator $ \mathcal{H} $ is a continuous on $ C^{+}(\Omega)$.
\end{lemma}

\begin{proof}
For any $ \phi, \psi \in C^{+}(\Omega)$,
\begin{equation}
\sup_{z\in\Omega} \left|(\mathcal{H}\phi)(z) - (\mathcal{H} \psi)(z) \right| \leq \bar{D} \int_\thom a(\theta) | q_{\infty}^{\phi}(\theta,\tilde{z}) - q_{\infty}^{\psi} (\theta,\tilde{z}) | \sigma(d\theta,d\tilde{z}) \label{AppC:Eq7}
\end{equation}
by Assumption (D). As $ q^{\phi}_{\infty} $ and $ q^{\psi}_{\infty} $ are fixed points of (\ref{AppC:Eq1}),
\begin{align*}
q^{\phi}_{\infty}(\theta,z) - q_{\infty}^{\psi}(\theta,z) & = \int_\th \left( s(\eta) - f(\psi(z),\eta)\right) \left( q^{\phi}_{\infty}(\theta,z) - q_{\infty}^{\psi}(\theta,z)\right) P^{\ast}(\theta,d\eta)  \\
& + \int_\th \left( f(\phi(z),\eta) - f(\psi(z),\eta) \right) \left( 1 - q_{\infty}^{\phi}(\eta,z)\right) P^{\ast}(\theta,d\eta).
\end{align*}
By Assumption (G) and (C),
\begin{align*}
\sup_{\theta \in \Theta }  \left| q^{\phi}_{\infty}(\theta,z) - q_{\infty}^{\psi}(\theta,z) \right| & \leq \sup_{\theta \in \Theta }\int_\th s(\eta) \left| q^{\phi}_{\infty}(\theta,z) - q_{\infty}^{\psi}(\theta,z)\right| P^{\ast}(\theta,d\eta)   + L  \left| \phi(z) - \psi(z)\right| \\
& \leq  \left(\sup_{\theta \in \Theta} s(\theta) \right) \left( \sup_{\theta\in\Theta} \left| q^{\phi}_{\infty}(\theta,z) - q_{\infty}^{\psi}(\theta,z)\right| \right) + L|\phi(z) - \psi(z)|.
\end{align*}
By Assumption (I),
\begin{equation}
\sup_{\theta \in \Theta } \left| q^{\phi}_{\infty}(\theta,z) - q_{\infty}^{\psi}(\theta,z) \right|  \leq  L \left( 1 - \sup_{\theta \in \Theta} s(\theta) \right) |\phi(z) - \psi(z)|.
\end{equation}
Substituting this into inequality (\ref{AppC:Eq7}), we see that $ \mathcal{H} $ is a (Lipschitz) continuous operator on $ C^{+}(\Omega)$.
\end{proof}

\begin{lemma}[Order compactness]
For any $ \psi_{1}, \psi_{2} \in C^{+}(\Omega) $, $ \mathcal{H} $ maps the set $ \{\phi \in C^{+}(\Omega) : \psi_{1} \leq \phi \leq \psi_{2} \} $ to a relatively compact set.
\end{lemma}

\begin{proof}
As $ q_{\infty}^{\phi} (\theta,z) \leq 1 $ for all $ (\theta,z) \in \Theta\times\Omega$ and $ D(\cdot,\cdot) $ is uniformly bounded and equicontinuous by Assumption (D), the image of  $ \{\phi \in C^{+}(\Omega) : \psi_{1} \leq \phi \leq \psi_{2} \} $  under $ \mathcal{H} $ is a set of uniformly continuous functions. Order compactness now follows from the Arzel\`{a}-Ascoli theorem. 
\end{proof}

We may now apply the cone limit set trichotomy. We have seen in the proof of strong sublinearity that for any $ \phi \in C^{+}(\Omega) $, $ q_{\infty}^{\phi}(\theta,z) < \sup_{\eta\in\Theta} s(\eta) $ for all $ (\theta,z) \in \Theta \times \Omega $. Substituting this bound into equation (\ref{AppC:Eq3}) shows that $ \mathcal{H} $ is a bounded operator. This excludes the possibility of an orbit of $ \mathcal{H} $ being unbounded. Therefore, either (i) each orbit of $ \mathcal{H} $ converges to 0, the unique fixed point of $ \mathcal{H}$, or (ii) each nonzero orbit converges to $ q^{\ast} \gg 0$; the unique nonzero fixed point of $\mathcal{H}$. This completes the proof of Theorem \ref{thm:CLST}.
\end{proof}

To determine which of the two possibilities from Theorem \ref{thm:CLST} occurs, we need to define the operator $ \mathcal{A} : C^{+}(\Omega) \rightarrow C^{+}(\Omega) $,
\[
(\mathcal{A} \phi)(z) = \int_\thom a(\tilde{\theta}) D(z,\tilde{z}) \phi(\tilde{z}) \sum_{m=1}^{\infty} \mathbb{E} \left\{f^{\prime}(0;\theta^{\ast}_{m}) \prod_{n=1}^{m-1} s(\theta^{\ast}_{n}) \mid \theta_{0}^{\ast} = \tilde{\theta} \right\} \sigma(d\tilde{z},d\tilde{\theta}).
\]
The spectral radius of $ \mathcal{A}$ is denoted $ r(\mathcal{A}) $. 

\begin{theorem} \label{thm:Threshold}
Suppose Assumptions (A)-(D), (G)-(J) hold. If $ r(\mathcal{A}) \leq 1 $, then $ 0 $ is the unique fixed point of the recursion (\ref{theorem1:Eq5}). If $ r(\mathcal{A}) > 1 $, then recursion (\ref{theorem1:Eq5}) has a non-zero fixed point. 
\end{theorem}

\begin{proof}
By Assumptions (C), (G) and (H)
\[
\phi(z) \mathbb{E} \left\{f^{\prime}(0;\theta^{\ast}_{m}) \prod_{n=1}^{m-1} s(\theta^{\ast}_{n}) \mid \theta_{0}^{\ast} = \theta \right\}  - q_{\infty}^{\phi}(\theta,z) \geq \mathbb{E} \left\{\phi(z) f^{\prime}(0;\theta^{\ast}_{1}) - f(\phi(z);\theta_{1}^{\ast}) \mid \theta_{0}^{\ast} = \theta \right\}
\]
so 
\[
(\mathcal{A}\phi)(z) - (\mathcal{H} \phi)(z) \geq \int_\thom a(\tilde{\theta} ) D(z,\tilde{z})  \mathbb{E} \left\{\phi(z) f^{\prime}(0;\theta^{\ast}_{1}) - f(\phi(z);\theta_{1}^{\ast}) \mid \theta_{0}^{\ast} = \tilde{\theta} \right\} \sigma(d\tilde{z},d\tilde{\theta}).
\]
Now $ \phi(z) f^{\prime}(0;\theta) - f(\phi(z);\theta) > 0 $ if $ \phi(z) > 0 $ and $ \theta \in \Theta\backslash \Theta_{1} $.  By Assumption (D), $ D(z,\tilde{z}) > 0 $ for all $ (z,\tilde{z}) \in \Omega \times \Omega $. Assumption (J) now implies that if $ \phi \neq 0 $, then
\[
(\mathcal{A} \phi)(z) - (\mathcal{H} \phi )(z) > 0,
\]
for all $ z \in \Omega $. If $ \phi = 0 $, then by substitution $ \mathcal{A} \phi = \mathcal{H} \phi = 0 $. By \cite[Lemma A.1]{MP:14}, it follows that if $ r(\mathcal{A}) \leq 1 $, then $ 0 $ is the unique fixed point of $ \mathcal{H}$ and hence is the unique fixed point of the recursion (\ref{Thm2:Eq1})-(\ref{Thm2:Eq2}).

As in \cite[Lemma A.2]{MP:14}, if $ r(\mathcal{A}) > 1 $, then there exists a $ \psi \in C^{+}(\Omega) $ such that $ \psi \leq \mathcal{H} \psi $. This implies that the set $ \{\phi \in C^{+}(\Omega) : \psi \leq \phi \} $ is invariant under $ \mathcal{H} $. The Schauder fixed point theorem \cite[Theorem 5.1.2]{Istratescu:81} shows that $ \mathcal{H} $ has a fixed point in $ \{\phi \in C^{+}(\Omega) : \psi \leq \phi \} $. This non-zero fixed point must be unique by Theorem \ref{thm:CLST}.
\end{proof}

\begin{theorem} \label{thm:Conv}
Suppose Assumptions (A)-(D), (G)-(J) hold. If $ r(\mathcal{A}) \leq 1 $, then all trajectories of  the recursion (\ref{theorem1:Eq5}) converge to $ 0 $. If $ r(\mathcal{A}) > 1$ and $ \frac{\partial \mu_{0}}{\partial \sigma} (\theta,z) > \epsilon $ for all $ (\theta,z) \in \Theta\times\Omega $ and some $ \epsilon > 0 $, then trajectory of the recursion (\ref{theorem1:Eq5}) converge to the non-zero fixed point. 
\end{theorem}

\begin{proof}
Under Assumptions (C) and (G), we can use a similar argument to that used in the proof of Lemma \ref{Lem:Mono} to show that recursion (\ref{theorem1:Eq5}) has the monotonicity property; if $ \frac{\partial \tilde{\mu}_{0}}{\partial \sigma} $ and  $ \frac{\partial \mu_{0}}{\partial \sigma} $ are two initial conditions such that $ \frac{\partial \tilde{\mu}_{0}}{\partial \sigma} \leq \frac{\partial \mu_{0}}{\partial \sigma} $ in the partial ordering on $ C^{+}(\Omega) $, then under recursion (\ref{theorem1:Eq5}) $ \frac{\partial \tilde{\mu}_{t}}{\partial \sigma} \leq \frac{\partial \mu_{t}}{\partial \sigma} $ for all $ t \geq 0 $. The proof then follows the arguments of \cite[Theorems 5.1 and 5.3 ]{BIT:91} as used in \cite[Theorem 3]{MP:13}.
\end{proof}

\begin{proof}[Proof of Theorem \ref{theorem3}]
The result follows from Theorems \ref{thm:CLST}, \ref{thm:Threshold} and \ref{thm:Conv} with $ \sigma $ replaced by the product measure $ \zeta \times \pi $.
\end{proof}

\begin{proof}[Proof of Theorem \ref{thm:LandDist}] Assuming $ f(x;\theta) = s(\theta)\bar{f}(x) $, equation (\ref{AppC:Eq2}) becomes
\begin{equation}
q^{\phi}_{\infty}(\theta,z) = \sum_{m=1}^{\infty} \bar{f}(\phi(z)) (1 - \bar{f}(\phi(z)))^{m-1} \mathbb{E} \left\{ \prod_{n=1}^{m} s(\theta^{\ast}_{n}) \mid \theta_{0}^{\ast} = \theta \right\}. \label{AppC:Eq2b}
\end{equation}
Integrating $ a(\theta) q^{\phi}_{\infty}(\theta,z) $ with respect to $ \pi $ shows that 
\begin{equation}
\int_\th a(\theta) q^{\phi}_{\infty}(\theta,z) \pi(d\theta) = \sum_{m=1}^{\infty} \bar{f}(\phi(z)) (1 - \bar{f}(\phi(z)))^{m-1} \mathbb{E} \left\{ a(\theta_{m+1}) \prod_{n=1}^{m} s(\theta_{n}) \right\}. \label{AppC:Eq2c}
\end{equation}
Substituting equation (\ref{AppC:Eq2c}) into equation (\ref{AppC:Eq3}) and using the product form for $ \sigma $ yields
\begin{equation}
\left( \mathcal{H} \phi\right) (z) = \int_\om  D(z,\tilde{z})  \sum_{m=1}^{\infty} \bar{f}(\phi(z)) (1 - \bar{f}(\phi(z)))^{m-1} \mathbb{E} \left\{ a(\theta_{m+1}) \prod_{n=1}^{m} s(\theta_{n}) \right\} \zeta(\tilde{z}) d\tilde{z}. \label{AppC:Eq3b}
\end{equation}
As $ \mathcal{H} $ depends on $ (\pi, P ) $ only through the sequence $ \mathbb{E} \left( a(\theta_{m+1})\prod_{n=1}^{m} s(\theta_{n}) \right),\ m \geq 1 $ so must its fixed point. 
\end{proof}

\begin{proof}[Proof of Theorem \ref{thm:SO}]
Let $ \tilde{\mathcal{H}} $ be the operator obtained by replacing $ \theta_{n} $ with $ \tilde{\theta}_{n} $ in equation (\ref{AppC:Eq3b}). From (\ref{Thm4:Eq1}), it follows that  for any $ \phi \in C^{+}(\Omega) $
\[
\tilde{\mathcal{H}} \phi \leq \mathcal{H} \phi.
\]
If $ \tilde{\phi}^{\ast} $ is a fixed point of $ \tilde{\mathcal{H}}$, then $ \tilde{\phi}^{\ast} = \tilde{\mathcal{H}}\tilde{\phi}^{\ast} \leq \mathcal{H}\tilde{\phi}^{\ast} $. Therefore, the set $ \{\phi \in C^{+}(\Omega): \tilde{\phi}^{\ast} \leq \phi \} $ is invariant under $ \mathcal{H} $ and by the Schauder fixed point theorem $ \mathcal{H} $ has a fixed point in this set. As $ \mathcal{H} $ has a unique non-zero fixed point (Theorem \ref{thm:CLST}), $ \tilde{\phi}^{\ast} \leq \phi^{\ast}$. Now by equation (\ref{AppC:Eq2c})
\begin{align*}
\int_\th a(\theta) \tilde{q}^{\ast} (\theta,z) \tilde{\pi}(d\theta)  & =  \sum_{m=1}^{\infty} \bar{f}(\phi(z)) (1 - \bar{f}(\tilde{\phi}^{\ast}(z)))^{m-1} \mathbb{E} \left\{ a(\tilde{\theta}_{m+1}) \prod_{n=1}^{m} s(\tilde{\theta}_{n}) \right\} 
\\
& \leq   \sum_{m=1}^{\infty} \bar{f}(\phi(z)) (1 - \bar{f}(\tilde{\phi}^{\ast}(z)))^{m-1} \mathbb{E} \left\{ a(\theta_{m+1}) \prod_{n=1}^{m} s(\theta_{n}) \right\}
\end{align*}
from inequality (\ref{Thm4:Eq1}).  Now let $ \alpha_{m} = \mathbb{E} \left( a(\theta_{m+1})\prod_{n=1}^{m} s(\theta_{n}) \right)  - \mathbb{E} \left(a(\theta_{m+2}) \prod_{n=1}^{m+1} s(\theta_{n}) \right) $. Then
\begin{align*}
\int_\th a(\theta) \tilde{q}^{\ast} (\theta,z) \tilde{\pi}(d\theta)  & \leq \sum_{m=1}^{\infty} \bar{f}(\tilde{\phi}^{\ast}(z)) (1 - \bar{f}(\tilde{\phi}^{\ast}(z)))^{m-1} \left\{ \sum_{r=m}^{\infty} \alpha_{r} \right\}\nonumber\\
& =  \sum_{r=1}^{\infty}  (1 - (1-\bar{f}(\tilde{\phi}^{\ast}(z)))^{r})  \alpha_{r}
\end{align*}
As $ \tilde{\phi}^{\ast} \leq \phi^{\ast}$ and $ \bar{f} $ is increasing, $ (1 - (1-\bar{f}(\tilde{\phi}^{\ast}(z)))^{r}) \leq (1 - (1-\bar{f}(\phi^{\ast}(z)))^{r}) $ for all $ r \geq 1 $ and all $ z \in \Omega $. Therefore, 
\begin{align*}
\int_\th a(\theta) \tilde{q}^{\ast} (\theta,z) \tilde{\pi}(d\theta)  & \leq  \sum_{r=1}^{\infty}  (1 - (1-\bar{f}(\phi^{\ast}(z)))^{r})  \alpha_{r}  \\
& =   \sum_{m=1}^{\infty} \bar{f}(\phi(z)) (1 - \bar{f}(\phi^{\ast}(z)))^{m-1} \mathbb{E} \left\{ a(\theta_{m+1}) \prod_{n=1}^{m} s(\theta_{n}) \right\} \\
& = \int_\th a(\theta) q^{\ast} (\theta,z) \pi(d\theta).
\end{align*}
\end{proof}

\begin{proof}[Proof of Lemma \ref{thm:Max}]
From the general form of H\"{o}lder's inequality
\[
\mathbb{E} \left( \prod_{n=0}^{m} s(\theta_{n})\right) \leq
\prod_{n=0}^{m} \mathbb{E} \left(s(\theta_{n})^{m+1}\right)^{1/(m+1)}.
\]
As $ (\theta_{t},\ t\geq 0) $ is assumed stationary $ \mathbb{E}(s(\theta_{t})^{m+1}) = \mathbb{E}(s(\theta_{0})^{m+1}) $ for $ t = 1,\ldots,m$. Therefore,
\[
\mathbb{E}\left(\prod_{n=0}^{m}s(\theta_{n})\right) \leq \mathbb{E}(s(\theta_{0})^{m+1}).
\]
\end{proof}

\subsection{Numerical approximation of the fixed point}

To compute $ \int_\th q^{\ast}(\theta,z) \pi(d\theta) $, we first numerically determined the non-zero fixed point the operator $ \mathcal{H} $ from (\ref{AppC:Eq3}) employing the product form of $ \sigma $ and the expression for $ \int_\th q^{\phi}(\theta,z) \pi(d\theta) $ given by (\ref{AppC:Eq2c}).  This was done by fixed point iteration of an approximation to the operator $ \mathcal{H} $ where we (i) approximated the integral with respect to $ \zeta $ by a Reimann sum with 500 terms, (ii) truncated the infinite sum in (\ref{AppC:Eq2c}) to 1000 terms, and (iii) approximated the moments $ \mathbb{E}(\prod_{n=1}^{m}s(\theta_{n})) $ by simulating 1000 sample paths of the survival process. The fixed point of $ \mathcal{H} $  was then substituted into equation  (\ref{AppC:Eq2c}) to give an approximation to the limiting probability of the patch being occupied.

\end{document}